\documentclass[a4paper]{article}
\usepackage[dvips]{graphics,graphicx}
\usepackage{amsfonts,amssymb,amsmath,color,mathrsfs, amstext, bm}
\usepackage{amsbsy, amsopn, amscd, amsxtra, amsthm,authblk}
\usepackage{enumerate,algorithmic,algorithm}
\usepackage{mathtools}
\usepackage{color}
\usepackage{svg}
\usepackage{appendix}

\usepackage{subcaption} 
\usepackage{epstopdf}

\usepackage{geometry}
\usepackage{hyperref}
\hypersetup{hypertex=true,colorlinks=true,
linkcolor=blue,
anchorcolor=red,
citecolor=red}

\newcommand{\tcb}[1]{\textcolor{blue}{#1}}

\numberwithin{equation}{section}  
\newtheorem{lemma}{Lemma}[section]

\newtheorem{remark}{Remark}[section]

\newtheorem{theorem}{Theorem}[section]

\def\R{\mathbb{R}}

\def\cU{\mathcal{U}}

\def\cH{\mathcal{H}}
\def\cG{\mathcal{G}}

\def\bfp{\mathbf{p}}
\def\bfq{\mathbf{q}}

\def\r{\rangle}

\def\d{\mathrm{d}}

\begin{document}

\title{A generalized discontinuous Hamilton Monte Carlo for transdimensional sampling}

\author[a,b,c]{Lei Li\thanks{E-mail: leili2010@sjtu.edu.cn}}
\author[a]{Xiangxian Luo\thanks{E-mail: xiangxianluo@sjtu.edu.cn} }
\author[a]{Yinchen Luo\thanks{E-mail: luoyinchen@sjtu.edu.cn}}
\affil[a]{School of Mathematical Sciences, Shanghai Jiao Tong University, Shanghai, 200240, P.R.China.}
\affil[b]{Institute of Natural Sciences, MOE-LSC, Shanghai Jiao Tong University, Shanghai, 200240, P.R.China.}
\affil[c]{Shanghai Artificial Intelligence Laboratory}
\date{}

\maketitle

\begin{abstract}
In this paper, we propose a discontinuous Hamilton Monte Carlo (DHMC) to sample from dimensional varying distributions, and particularly the grand canonical ensemble. The DHMC was proposed in [Biometrika, 107(2)] for discontinuous potential where the variable has a fixed dimension. When the dimension changes, there is no clear explanation of the volume-preserving property, and the conservation of energy is also not necessary. We use a random sampling for the extra dimensions, which corresponds to a measure transform. We show that when the energy is corrected suitably for the trans-dimensional Hamiltonian dynamics, the detailed balance condition is then satisfied. For the grand canonical ensemble, such a procedure can be explained very naturally to be the extra free energy change brought by the newly added particles, which justifies the rationality of our approach. To sample the grand canonical ensemble for interacting particle systems, the DHMC is then combined with the random batch method to yield an efficient sampling method. In experiments, we show that the proposed DHMC combined with the random batch method generates samples with much less correlation when compared with the traditional Metropolis-Hastings method.
\end{abstract}

\section{Introduction}\label{sec:introduction}

The computation of statistics, such as the pressure and energy, is crucial for determining the physical properties of a system and plays a central role in statistical physics and chemistry \cite{frenkel2001understanding}. However, performing the integrals against the corresponding Gibbs distribution is prohibitive, so one often turns to the Monte Carlo methods \cite{frenkel2001understanding,rubinstein2016simulation}, i.e., drawing samples from the target distribution and estimating the statistics by the empirical average using the samples.
The Markov chain Monte Carlo (MCMC) methods \cite{gamerman2006markov,hastings1970monte} are a class of popular Monte Carlo methods for drawing samples approximately from the desired distribution, and the Metropolis-Hastings \cite{metropolis1953,hastings1970monte} algorithm is one of the most popular examples.
When treating with the distributions for interacting particle systems, direct application of the Metropolis-Hastings algorithm requires the computation of the difference of the energy for the acceptance rate, which is usually not very efficient because moving $O(1)$ particles usually needs $O(N)$ computational cost.  Using the splitting Monte Carlo, where the proposal is generated by some dynamics like the Langevin dynamics and the Hamiltonian dynamics, could be turned into efficient approximate sampling methods when combined with the random batch strategies \cite{li2020random,li2023splittinghamiltonianmontecarlo}.

We are concerned with distributions where the variables are from domains with varying dimensions. One particular example is the grand canonical ensemble \cite{frenkel2001understanding}. The grand canonical ensemble describes an open system of particles in thermodynamic equilibrium with a reservoir, exchanging both energy and particles while volume, shape, and other external coordinates of the system remain constant. Inspired by \cite{li2023splittinghamiltonianmontecarlo}, we turn our attention to the Hamiltonian Monte Carlo (HMC) \cite{duane1987hybrid, neal2011mcmc, barbu2020hamiltonian} for sampling from such distributions. The HMC is a specific MCMC method and can be viewed as a special case of the Metropolis-Hastings algorithm where the proposal is given by a Hamilton dynamics. However, HMC draws on principles from classical mechanics, treating the parameters of the target distribution as particles moving within a potential energy landscape. By utilizing gradient information from the log-probability density, HMC generates proposals that are more informed than those from traditional random walk methods, allowing it to explore the parameter space effectively \cite{barbu2020hamiltonian,beskos2013optimal,duane1987hybrid,neal2011mcmc}.
The  HMC  has gained a lot of attention in practice due to its scalability and efficiency in high-dimensional settings \cite{neal2011mcmc, beskos2013optimal}. In HMC, to sample from a distribution 
\begin{gather}\label{target_dist}
\pi(\d \bfq)\propto \exp(-\beta U(\bfq))\,\d \bfq
\end{gather}
 where $\bfq\in \R^D$ and $\beta$ is a chosen (fixed) parameter whose physical significance is the inverse temperature, one adds conjugate variables $\bfp\in \R^D$ called the momentum variable. Then, one considers
\begin{gather}
H(\bfp, \bfq)=U(\bfq)+K(\bfp),
\end{gather}
for some suitable function $K: \R^D\to \R$. Usually, $K(p)=\frac{1}{2}\bfp^T M^{-1}\bfp$ for some positive definite matrix $M$. The following Markov chain is then constructed starting with $(\bfp_n, \bfq_n)$  so that the invariant measure is $\tilde{\pi}(\d \bfp\d \bfq)\propto \exp(-\beta H(\bfp, \bfq))\,\d \bfp \d \bfq$, with $\pi$ being its marginal.
\begin{enumerate}[(a)]
\item Draw a new $\tilde{\bfp}_n \sim Z_\bfp^{-1}\exp(-\beta K(\bfp))$.
\item Run the Hamilton dynamics
\begin{gather}
\left\{
\begin{split}
&\dot{\bfq}=\nabla_\bfp H,\\
&\dot{\bfp}=-\nabla_\bfq H,
\end{split}
\right.
\end{gather}
for some time $T$, and obtain $(\tilde{\bfp}_{n+1}, \bfq_{n+1})$.
\item Negate the momentum variable to obtain $(\bfp_{n+1}, \bfq_{n+1})$. Then, accept this new sample with the Metropolis step with acceptance rate $\min\{1, \exp(-\beta(H(\bfp_{n+1}, \bfq_{n+1})-H(\bfp_n, \bfq_n))) \}$.
\end{enumerate}
For time-continuous dynamics, due to the conservation of energy for Hamilton dynamics $H(\bfp_{n+1}, \bfq_{n+1})-H(\bfp_n, \bfq_n)=0$, so the new sample should be accepted. This rejection step would be considered if we discretize the dynamics. However, if we aim to improve efficiency in applications like the interacting particle system, we can accept the proposal anyway, even if there is discretization error, which often vanishes as the time step goes to zero, especially when combined with the random batch strategy \cite{li2023splittinghamiltonianmontecarlo}. Moreover, the negation of the momentum is necessary only in theory to ensure the detailed balance, which is not necessary in practice, as one would draw a new momentum variable. 

In many applications, the distributions may have discrete variables like the grand canonical ensemble. These can be viewed as distributions that are piecewise smooth but have discontinuities across different domains. To address the limitations of HMC in handling nonsmooth or distributions with discrete variables, the authors of \cite{Nishimura_2020} introduced the Discontinuous Hamiltonian Monte Carlo (DHMC), expanding the applicability of HMC to a broader range of problems. A key feature of DHMC to treat discrete variables $n$ is to embed them into a continuous space and the potential is now piecewise smooth. The dual variable for $n$ would have the following Laplace momentum 
\begin{gather}
K(p_n)=|p_n|/m.
\end{gather}
This is beneficial as $\frac{dn}{dt}$ depends only on the signs of $p_n$ and not on its magnitude. So the integrator can use a small number of target density evaluations to jump through multiple discontinuities while approximately preserving the total energy.  See section \ref{sec:review_DHMC} below for a more detailed introduction, and see the original paper \cite{Nishimura_2020} for a complete explanation.

Though the original DHMC treated the distributions with discontinuity, the variables across different domains have the same dimension. Many practical scenarios involve probability distributions with variable dimensionality.  For instance, in statistical mechanics, systems that exchange particles and energy with a reservoir are described by the Grand Canonical Ensemble as mentioned above, resulting in a change of variable dimensions \cite{boinepalli2003grand,10.5555/547952}.  In the statistical literature, some problems involve an unknown number of model parameters. These are known as Bayesian model determination problems, which have a wide range of applications, such as change-point models, finite mixture models, variable selection, and Bayesian neural networks. For more details, readers are referred to \cite{rjmcmc}. Traditional sampling methods from such distributions include the
Reversible Jump Markov Chain Monte Carlo (RJMCMC) \cite{rjmcmc,23b6b6ec-42a0-3e0a-8bf2-dc861b85327c} and the transdimensional piecewise deterministic Markov processes \cite{doi:10.1080/01621459.2022.2099402}.


In this paper, we aim to propose an extended DHMC to treat distributions with varying dimensions.
This could be particularly useful in the case of the grand canonical ensemble for interacting particle systems, as we can combine it with the random batch strategy \cite{Jin_2020} to improve the efficiency as done in \cite{li2023splittinghamiltonianmontecarlo}. When the dimension changes, there is no clear explanation of the volume-preserving property of the Hamiltonian dynamics. We use a random sampling for the extra dimensions, which corresponds to a measure transform. There is no reason to require the conservation of energy during this jump, and we perform a certain correction for the energy. We show that when the energy is corrected suitably for the trans-dimensional Hamiltonian dynamics, the detailed balance condition is then satisfied. For the grand canonical ensemble, such a procedure can be explained very naturally to be the extra free energy change brought by the newly added particles, which justifies the rationality of our approach. 
We show in examples that this extended DHMC can effectively sample from target distributions where the dimension can vary while leveraging the advantages of Hamiltonian Monte Carlo.

The remainder of this paper is organized as follows.  In Section \ref{sec:review_DHMC} we give a brief review of discontinuous Hamiltonian Monte Carlo. Section \ref{sec:general_DHMC} introduces the general framework of our algorithm. In section \ref{section:discretization}, we discuss some details of the discretization and situate our algorithm in the framework of RJMCMC \cite{rjmcmc}. In section \ref{sec:grandcanonical} we apply our proposed algorithm to the simulation of the grand canonical ensemble and propose a version for the interacting particle systems that combines the potential splitting and the Random batch strategy to reduce computational complexity. In section \ref{sec:numerical} we present some numerical results to validate our method.

\section{A Brief Review of the DHMC}\label{sec:review_DHMC}
In this section, we give a brief introduction to the DHMC in \cite{Nishimura_2020} and refer the readers to the original paper for more details. 

The DHMC is based on a theory of discontinuous Hamiltonian dynamics. Suppose that the potential function $U$ in \eqref{target_dist} is piecewise smooth. To find the correct Hamiltonian dynamics under this potential function, one may use a sequence of smooth functions  $U_{\delta}(\mathbf{q'})=\int U(\mathbf{q})\phi_{\delta}(\mathbf{q'}-\mathbf{q})\text{d}\mathbf{q}$ to approximate $U$, where $\phi$ is a mollifier and take the limit of the dynamics as $\delta\to 0$. The behavior of the limiting dynamics at the discontinuity is as follows. 
Let $\nu$ denote the normal vector of the smooth hypersurface formed by the discontinuity points. When a trajectory $(\mathbf{q}(t), \mathbf{p}(t))$ of the dynamics intersects the discontinuity at time $t$, one then considers the orthogonal decomposition of the momentum:
\[
\mathbf{p}(t^-) = \mathbf{p}_{\parallel}(t^-) + \mathbf{p}_{\perp}(t^-),
\]
where $\mathbf{p}_{\parallel}(t^-)$ and $\mathbf{p}_{\perp}(t^-)$ are the tangential and normal components, respectively. The tangential component remains unchanged, while the normal component may vary. 
Consider the potential jump at the interface
\[
\Delta U = \lim_{\epsilon \to 0} \left[ U(\mathbf{q}(t^+) + \epsilon \mathbf{p}(t^-)) - U(\mathbf{q}(t^-)) \right].
\]  
If $\Delta U > K(\mathbf{p}(t^-)) - K(\mathbf{p}_{\parallel}(t^-))$, then the normal component of the momentum reverses direction,  
$$
\mathbf{p}_{\perp}(t^+) = -\mathbf{p}_{\perp}(t^-).
$$  
Otherwise, the direction of the normal component is not changed and the magnitude of $\mathbf{p}_{\perp}$ decreases to satisfy total energy conservation 
\[
K(\mathbf p(t^-))+U(\mathbf q(t^-))=K(\mathbf p(t^+))+U(\mathbf q(t^+)).
\]
Intuitively, the trajectory crosses the discontinuity if the kinetic energy associated with the normal component of the momentum is sufficient to overcome the potential energy gap. Otherwise, it reflects back from the hypersurface.

In the presence of $d'$ discrete parameters, one then embeds them into a continuous space by introducing latent variables $n\in\R^{d'}$ so that the distribution has a piecewise smooth density. Often, this embedding gives piecewise constant functions with respect to $n$.  For the discontinuous Hamiltonian dynamics associated with $n$, Nishimura et al. \cite{Nishimura_2020} proposed using a Laplace distribution for the momentum variable, given by  $K_n(p_n)\propto e^{-\|p_n\|_1/ m_n}$. Under this choice, the Hamilton equations corresponding to $n$ and $p_n$ are given by
\begin{equation}
    \dot{n}=\frac{1}{m_n}\text{sign}(p_n),\quad \dot{p}_n = -\nabla_n U(q, n). 
\end{equation}
Here, $\mathrm{sign}(p_n)$ is the componentwise signs of $p_n$. Note that in applications, $U(q, n)$ is often piecewise constant with respect to $n$, so $\nabla_n U(q, n)$ is zero away from the discontinuity. Across the discontinuity, the change of $p_n$ is then given by the discontinuous Hamiltonian dynamics above.  Notably, the evolution of  $n$ is governed by the sign of the momentum $p_n$, rather than its magnitude. This implies that if the sign of  $p_n$ does not change during time interval $[t,t+\Delta t]$, the trajectory can be computed exactly, ensuring exact preservation of the total energy. As demonstrated in \cite{Nishimura_2020}, this property is crucial for enabling the system to traverse multiple discontinuities while approximately preserving total energy, requiring only a single evaluation of the potential.

\section{Generalized Discrete Hamilton Monte Carlo for Trans-dimensional Sampling}\label{sec:general_DHMC}

In this section, we propose a general framework of the Discrete Hamilton Monte Carlo (DHMC) for dimension-variable probability distributions. 

\subsection{The effective Hamiltonian}\label{subsec:effectiveHal}

Consider the target probability distribution with varying-dimensional state space.
\begin{equation}
\label{target_dis}
\pi(\mathbf{q}^N, N)=\frac{1}{Z_{\beta}}e^{-\beta\mathcal{U}_N(\mathbf{q}^N)}\quad\text{where   } N\in \mathbb N.
\end{equation}
here $Z_{\beta} = \sum_{N\in \mathbb N_+}\int e^{-\beta\mathcal{U}_N(\mathbf{q}^N)}d\mathbf{q}^N$ is the normalization constant.
Here $\mathbf{q}^N\in V_{d_N} \subset \mathbb R^{d_N}$ is of dimension $d_N$. 
Without losing generality, we assume that $d_N$ is strictly  monotonically increasing with respect to $N$, which means that $N<M\iff d_N<d_M$. 
Here, we will require that the projection of $V_{d_{N+1}}$ onto the first $d_{N}$ coordinates will be $V_{d_N}$.

Similar to the original DHMC, we first embed the discrete indicator $N$ into the continuous real number axis by associating a parameter $n$.
Then, $n\in [N, N+1)$ indicates that the index is $N$. 
The state space can be viewed as the union of space of different dimensions $S=\bigcup_N V_{d_N}$.  We use $\mathbf{q}$ to represent a point in $S$.
Define the new probability distribution density as
\begin{gather}
\hat{\pi}(\mathbf{q},n)=\sum_{N\ge 0}\pi(\mathbf{q}, N)\mathbf{1}_{\{N \leq n<N+1\}}.
\end{gather}
\begin{remark}
The range for $n$ is $[0,\infty)$. This is equivalent to requiring that $\hat{\pi}=0$ for $n<0$ or the effective potential to be $+\infty$
for $n<0$. Note that $\hat{\pi}$ is not defined on $S\times [0,\infty)$. Instead, it is defined on $\bigcup_N V_{d_N}\times [N, N+1)$.
\end{remark}
Then we introduce conjugate variables, $p_n\in\mathbb R$ and $\mathbf{p}\in \bigcup_N \mathbb R^{d_N}$ to $n$ and $\mathbf{q}$ respectively. Then, the joint distribution of $(\mathbf{q},\mathbf{p},n,p_n)$ is defined as follows.
\begin{gather}\label{formal-dist}
\bar{\pi}(\mathbf{q},\mathbf{p},n,p_n)=\hat{\pi}(\mathbf{q},n)\pi_{p}(\mathbf{p}, n)\pi_{p_n}(p_n)
\end{gather}
where 
\begin{gather}
\begin{split}
& \pi_{p}(\mathbf{p}, n) =\sum_{N\geq0} Z_N^{-1}e^{-\beta K(\mathbf p)}\mathbf{1}_{\{N \leq n<N+1\}},\quad K(\mathbf{p}) = \frac{1}{2m}\|\mathbf{p}\|^2,\quad Z_N=\left(\frac{2m\pi}{\beta}\right)^{d_N/2}, \\
& \pi_{p_n}(p_n)=Z_n^{-1}e^{-\beta K_{n}(p_n)},\quad K_{n}(p_n) = \frac{|p_n|}{m_n},\quad Z_n=\frac{\beta}{2m_n}.
\end{split}
\end{gather}
Here, the $\pi$ in $Z_N$ is the ratio of a circle's circumference to its diameter, which is a constant. This should not be confused with the target distribution in \eqref{target_dis}. Note that $ \pi_{p}(\mathbf{p}, n)$ is not a probability measure on its own and could be understood as a conditional probability measure of $\mathbf{p}$. However, when combined with  $\hat{\pi}$ and $\pi_{p_n}$, the resulting distribution $\bar\pi$ is a valid probability measure. Obviously, the marginal distribution with respect to $\mathbf{q}$ is the target distribution. 

The Hamiltonian of this system is then given by
\begin{multline}\label{Hamiltonina}
\mathcal{H}(\mathbf{q},\mathbf{p},n,p_n)=-\frac{1}{\beta}\log (\bar{\pi}(\mathbf{q},\mathbf{p},n,p_n))\\
= \sum_{N\ge 0} \left( \mathcal{U}_N(\mathbf q)+K(\mathbf p)  + \frac{d_N}{2\beta}\log \left(\frac{2m\pi}{\beta}\right)\right)\mathbf{1}_{\{N \leq n<N+1\}}+K_n(p_n)+C,
\end{multline}
\tcb{}
which is a piecewise smooth function with discontinuity at integer values of $n$.
In this setting, $K(p)$ and $K_n(p_n)$ are the  ``kinetic energies", which are Gaussian and the Laplace, respectively.
One may understand 
\[
\sum_{N\ge 0} \left( \mathcal{U}_N(\mathbf q)+K(\mathbf p) \right)\mathbf{1}_{\{N \leq n<N+1\}}+K_n(p_n)
\]
as the energy, while $\sum_N \frac{d_N}{2\beta}\log (\frac{2m\pi}{\beta}) \mathbf{1}_{\{N \leq n<N+1\}}$
as the extra free energy associated with the dimension change.

\subsection{The Generalized Discontinuous Hamiltonian Monte Carlo}\label{subsec:dhmc}

We now consider constructing the Hamiltonian dynamics for $(\mathbf{q},\mathbf{p},n,p_n)$.
When $n\in (N, N+1)$ for some integer $N$, the $(\mathbf{q},\mathbf{p},n,p_n)$ system follows the Hamiltonian equation
\begin{equation}
\label{hmc_1}
    \begin{dcases}
        &\frac{\d\mathbf{q}}{\d t}=\nabla_{\mathbf{p}}\mathcal{H} = \frac{1}{m}\mathbf{p}\\
        &\frac{\d\mathbf{p}}{d t}=-\nabla_{\mathbf{q}}\mathcal{H} = -\nabla_{\mathbf{q}}\mathcal{U}_N(\mathbf{q})\\
        &\frac{\d n}{\d t}= \nabla_{p_n}\mathcal H = \frac{\text{sign}(p_n)}{m_n}\\
        &\frac{\d p_n}{\d t}= -\nabla_{n}\mathcal H = 0\\
    \end{dcases}
\end{equation}
Here, we make the convention that $\mathrm{sign}(0)=1$ so that $n$ would change even if $p_n=0$ (this has zero probability though).
Now consider that $n$ hits the discontinuity, i.e., when $n(t'-)=N+1$ for some $t'$ (which means that $n$ hits $N+1$ from left at time $t'$) or $n(t'-)=N$
(which means that $n$ hits $N$ from right at some time $t'$).

We consider $n(t'-)=N+1$ first, in which case $p_n>0$. If $n$ successfully transits into $[N+1, N+2)$, the dimension will change from $d_N$ to $d_{N+1}$. Unlike the DHMC for fixed dimensions, we must determine the components for the extra dimensions.
We will generally assume that the new variables are given by 
\[
\tilde{ \mathbf{q}}=\big[\mathbf{q},\mathbf{q}_{\mathrm{new}}\big],\mathbf{\tilde p}=\big[\mathbf{p},\mathbf{p}_{\mathrm{new}}\big].
\] 
Similar to DHMC, when the indicator $n$ encounters discontinuities, the algorithm calculates the energy difference required to traverse these discontinuities and compares it with the current kinetic energy relative to $p_n$. The energy difference $\Delta \cH$ here is stochastic due to the sampling of the new parameter.  Let us construct the dynamics across this discontinuity. We will sample $(\mathbf{q}_{\mathrm{new}},\mathbf{p}_{\mathrm{new}})$ from $\pi_{N,N+1}(\cdot,\cdot\ ;\textbf{q},\textbf{p})$. With this choice, the volume element has changed, according to the following
\begin{gather}\label{eq:measuretransform}
\d\mathbf{q}\, \d\mathbf{p}\to  \pi_{N,N+1}(\mathbf{q}_{\mathrm{new}},\mathbf{p}_{\mathrm{new}}; \mathbf{q}, \mathbf{p}) \d\mathbf{q}\,\d\mathbf{q}_{\mathrm{new}}\d\mathbf{p}\, \d\mathbf{p}_{\mathrm{new}}.
\end{gather}
This type of measure transformation does not quite agree with the usual volume-preserving property of Hamiltonian dynamics for the fixed dimension case. However, it satisfies
\[
\iint_{A}\d\mathbf{q}\, \d\mathbf{p}=\iint_{\bar{A}}\pi_{N,N+1}(\mathbf{q}_{\mathrm{new}},\mathbf{p}_{\mathrm{new}};\mathbf{q}, \mathbf{p})\, \d\tilde{\mathbf{q}}\, \d\tilde{\mathbf{p}}
\]
for all $A\in\mathcal{B}( V_{d_N}\times \R^{d_N})$ and $\bar{A}=\{ (\bfq, \bfq_{new},\bfp, \bfp_{new} )\in V_{d_{N+1}}\times \R^{d_{N+1}}: (\bfq,\bfp)\in A \}$. Taking this into consideration, this measure transformation can somehow be an analogue of the volume-preserving property. However, as the standard volume element is $\d\tilde{\mathbf{q}}\, \d\tilde{\mathbf{p}}=\d\mathbf{q}\, \d\mathbf{q}_{\mathrm{new}} \d\mathbf{p}\, \d\mathbf{p}_{\mathrm{new}}$,  we must compensate this in the Hamiltonian. The correction term $\delta \mathcal{G}$ (extra "free energy gap") resulted from the random sampling of new parameters is then given by
\begin{equation}
\delta \mathcal{G}:=\frac{1}{\beta}\text{log}(\pi_{N,N+1}(\mathbf{q}_{\mathrm{new}},\mathbf{p}_{\mathrm{new}}; \mathbf{q}, \mathbf{p})).
\end{equation}
In fact, due to the conservation of $\cH$, the value of $\pi(q, p, n, p_n)$ will be unchanged after crossing the discontinuity. The extra sampling density $\pi_{N, N+1}$, however, appears in the integration, so we must introduce the correction to cancel this density. In particular, we define the free energy barrier 
\begin{gather}\label{eq:gap1}
\begin{split}
&\delta \mathcal{H}^+:=\Delta \cH^+ +\delta \mathcal{G}\\
&=\Delta U^+ + \Delta K^+ +  \frac{d_{N+1}-d_N}{2\beta}\log(\frac{2m\pi}{\beta}) + \frac{1}{\beta}\text{log}(\pi_{N,N+1}(\mathbf{q}_{\mathrm{new}},\mathbf{p}_{\mathrm{new}}; \mathbf{q}, \mathbf{p})).
\end{split}
\end{gather}
Here, $\Delta\cH^+$ indicates the Hamiltonian jump if we increase the dimension. Then, the extended DHMC across the interface is then finished similarly as the original DHMC. In particular, if the kinetic energy of $p_n$ is large enough to overcome to the free energy barrier, then the jump succeeds and the direction of $p_n$ does not change, and the magnitude will adjust due to energy conservation. If $K_{n}(p_n)$ is smaller than the barrier, then the movement of indicator $n$ will bounce back, which means that the sign of $p_n$ will flip, and the newly added parameter will be deleted.

When $n(t'-)=N$, we simply take 
\begin{gather}
\bar{\mathbf{q}}\leftarrow\mathbf{q}|_1^{d_{N-1}},\quad \bar{\mathbf{p}}\leftarrow\mathbf{p}|_1^{d_{N-1}}.
\end{gather}
where $\mathbf{q}|_1^{d_{N-1}}$ indicates the first $d_{N-1}$ components of $\mathbf{q}$. To find the energy barrier, we again 
take into the reference measure transform in \eqref{eq:measuretransform}, and then 
\begin{gather}\label{eq:gap2}
\delta \cH^- :=\Delta \cH^- -\frac{1}{\beta}\log(\pi_{N-1,N}(\mathbf{q}|_{d_{N-1}+1}^{d_{N}}, \mathbf{p}|_{d_{N-1}+1}^{d_{N}}; \bar{\mathbf{q}}, \bar{\mathbf{p}})).
\end{gather}
where $\bfq|_{d_{N-1}+1}^{d_N}$ indicates the vectors made by the components from the $(d_{N-1}+1)$th one to the $d_N$th (or the last $d_N-d_{N-1}$ components) of $\bfq$. Similarly, $\Delta\cH^-$ indicates the Hamiltonian jump if we decrease the dimension. 
\begin{remark}
We remark that if $N=0$ and $n(t'-)=0$, one has $\Delta \cH=+\infty$ so $n$ will stay in $[0,\infty)$.
\end{remark}

Here, we make agreement that we may use the generic symbol $\delta \cH$ to represent either $\delta\cH^+$
or $\delta\cH^-$ and the concrete meaning should be clear in the context. If the concrete meaning needs to be emphasized, we then use $\delta\cH^+$ or $\delta\cH^-$. Similarly, we will use $\Delta\cH$ to mean the Hamiltonian jump if the change of dimension is clear.

If we fix the evolution time, the resulting Markov chain may not have good ergodicity properties. Instead, we follow the work of \cite{Nishimura_2020} to adopt random evolution time $T\sim\mathrm{UNIF}[T_\text{min},T_\text{max}]$. Given parameter $(\bfq, n)$, we first sample the auxiliary parameter $(\bfp,p_n)$ and then evolve the $(\mathbf{q},\mathbf{p},n,p_n)$ system for random time $T$.
We show the details in Algorithm \ref{DHMC}.

\begin{algorithm}[ht]
    \caption{general DHMC algorithm} \label{DHMC}
    \begin{algorithmic}
     \REQUIRE   $(\mathbf q,n)$, $N_{\text{sample}}$, $T_{\text{min}}$, $T_{\text{max}}$.
    \FOR{i=1,2,\ldots,$N_{\text{sample}}$}
    \STATE Sample $\mathbf{p}$ from $\pi_p(\cdot)$ and sample $p_n$ from $\pi_{n}(\cdot)$
    \STATE Sample $T$ from $\mathrm{UNIF}[T_{\text{min}},T_{\text{max}}]$
     \STATE Simulate dynamics of $(\mathbf{q},\mathbf{p},n,p_n)$ by Algorithm \ref{t-H dynamics} for time $T$ and obtain $(\mathbf{\tilde q},\mathbf{\tilde p},\tilde n,\tilde p_n)$
     \STATE update parameter $(\mathbf{q},n)\leftarrow( \mathbf{\tilde q},\tilde n)$
    \ENDFOR
    \end{algorithmic}
\end{algorithm}

\begin{algorithm}[ht]
    \caption{simulation of trans-dimensional Hamiltonian dynamics} \label{t-H dynamics}
    \begin{algorithmic}
     \REQUIRE   $(\mathbf q,\mathbf p,n,p_n)$ with $n\in [N, N+1]$, $T>0$ and initial time $t=0+$.
    \STATE Simulate dynamics of $(\mathbf{q},\mathbf{p},n,p_n)$ by System \eqref{hmc_1} upto $t=T$.
    \IF{For some $t'\in (0, T]$, $n(t'-)=N$ or $n(t'-)=N+1$}
     \IF{$n(t'-)=(N+1)$ (and thus $p_n>0$)}
       \STATE Sample $(\mathbf{q}_{\mathrm{new}},\mathbf{p}_{\mathrm{new}})$ from $\pi_{N,N+1}(\cdot)$ and propose $\tilde{\mathbf{q}}\leftarrow \big[\mathbf{q},\mathbf{q}_{\mathrm{new}}\big], \tilde{\mathbf{p}}\leftarrow \big[\mathbf{p},\mathbf{p}_{\mathrm{new}}\big]$
       \STATE Compute $\delta \cH(\mathbf{q},\mathbf{p}\to \tilde{\mathbf{q}},\tilde{\mathbf{p}})$ by \eqref{eq:gap1}.
     \ELSE
        \STATE Propose $\tilde{\mathbf{q}}\leftarrow\mathbf{q}|_1^{d_{N-1}},\tilde{\mathbf{p}}\leftarrow\mathbf{p}|_1^{d_{N-1}}$.
        \STATE  Compute $\delta \cH(\mathbf{q},\mathbf{p}\to \tilde{\mathbf{q}},\tilde{\mathbf{p}})$ by \eqref{eq:gap2}.
       \ENDIF
       \IF{$K(p_n)\geq \delta\cH$}
                 \STATE Accept the new point and update $\tilde p_n$ such that  $K(p_n)-\delta\cH=K(\tilde p_n)$, and set $t=t'+$.
        \ELSE
                \STATE Reject the new point and flip the sign of $p_n$ and set $t=t'+$.
       \ENDIF
     \ENDIF
    \end{algorithmic}
\end{algorithm}

Here, we make some extra comments about the method.
\begin{enumerate}[(a)]
\item If $\mathbf{q}_{\mathrm{new}}$ and $\mathbf{p}_{\mathrm{new}}$ are independent and  $\mathbf{q}_{\mathrm{new}}\sim\pi_q(\cdot)$, $\mathbf{p}_{\mathrm{new}}\sim \mathcal N(0,\frac{m}{\beta}I_{d_{N+1}-d_N})$, the free energy correction will cancel out the term relative to $\mathbf{p}$ so that
$$
\delta \cH= \Delta U + \frac{1}{\beta}\text{log}(\pi_{q}(\mathbf{q}_{\mathrm{new}}))
$$

\item When we increase the dimension or decrease the dimension, we do not have to append the newly added components or delete components at the end. We may add them or delete them at random locations. Moreover, some general dimension matching could be performed similarly as in RJMCMC literature (as discussed in chapter 3 of \cite{brooks2011handbook}), which is, however, not very suitable for our DHMC setting, especially for the grand canonical ensemble as in section \ref{sec:grandcanonical}. 

\item
Note that the dynamics considered in Algorithm \ref{DHMC} is time-continuous. In real simulations, we may discretize the dynamics
, and the change in $n$ may include multiple discontinuity points. In such cases, we may consider the total energy barrier $\delta\cH$ as an analogue to the total energy gap in the original DHMC method \cite{Nishimura_2020}. The correction $\delta \cG$ may be computed directly for the overall extra dimensions without really performing each jump. See section \ref{section:discretization} for more details.

\item 
To ensure that the dynamics are reversible, which can ensure the detailed balance condition discussed in chapter \ref{subsec:theory}, we require the distribution of the new parameters to be even with respect to $\mathbf{p}_{\mathrm{new}}$, that is,
\[ 
\pi_{N,N+1}(\mathbf{q}_{\mathrm{new}}, \mathbf{p}_{\mathrm{new}}; \mathbf{q}, \mathbf{p}) = \pi_{N,N+1}(\mathbf{q}_{\mathrm{new}}, -\mathbf{p}_{\mathrm{new}}; \mathbf{q}, -\mathbf{p}), \quad \forall \mathbf{q}_{\mathrm{new}}, \mathbf{p}_{\mathrm{new}}, \mathbf{q}, \mathbf{p}, N.
\]
\end{enumerate}

\subsection{Theoretical properties of DHMC}\label{subsec:theory}

In this section, we analyze the theoretical properties of DHMC. Especially, we will prove the detailed balance condition of the ideal DHMC algorithm. In classical HMC theory, the detailed balance condition results from the combined effects of reversibility, volume preservation, and energy conservation. However, in the problem discussed in this paper, due to the changing dimensions, these first two properties mentioned above do not directly hold. However, due to the construction of DHMC here, the properties are expected to hold as well.

For brevity, let $\varphi = (\textbf{q}, n),\psi = (\textbf{p},p_n)$. With slight abuse of notation, we also write $\varphi_N = (\textbf{q}, n)\in V_{d_N}\times [N,N+1),\psi_N = (\textbf{p},p_n)\in \R^{d_N}\times\R$, when $\dim(\textbf{q}) = \dim(\textbf{p})=d_N$.

As mentioned before, suppose we want to sample from a distribution $\hat\pi(\mathbf{q},n)$ with multiple modals of different dimensions.
When a trans-dimensional jump is proposed, to maintain reversible, 
we should have 
\begin{equation}
\int_{A\times B} \hat{\pi}(\d\varphi)\Gamma(\varphi\rightarrow \d\varphi')=\int_{B\times A} \hat{\pi}(\d\varphi')\Gamma(\varphi'\rightarrow \d\varphi),
\end{equation}
where $A\subset \mathbb{R}^{d_N}\times [N,N+1),B\subset\mathbb{R}^{d_{N'}}\times [N',N'+1)$ are arbitrary Borel sets and $\Gamma(\cdot,\cdot)$ is the transition kernel of ideal DHMC algorithm. Our goal in this subsection is to verify this for the ideal DHMC framework. Without loss of generality, we let
\begin{gather}
m_n=1.
\end{gather}
Let $\Gamma_T(\varphi_N,\psi_N\to \varphi_M, \psi_M)$ be the transition kernel of the DHMC dynamics for time length $T$.
By the well-known Markov property, the transition kernel of DHMC dynamics satisfies the Chapman-Kolmogorov equation
\begin{multline}\label{eq:ckexpansion}
\Gamma_{2T}(\varphi_N,\psi_N\to \varphi_M, \psi_M) \\
=\sum_{K=0}^\infty \int \Gamma_{T}(\varphi_N,\psi_N\to \varphi_K, \psi_K) \Gamma_{T}(\varphi_K,\psi_K\to \varphi_M, \psi_M) \d\varphi_K \d\psi_K.
\end{multline}
Let $\Gamma_T(\varphi_N\to\varphi_M)$ be the transition kernel of the ideal DHMC sampling algorithm for evolution time $T$.
By the construction of DHMC, each step, $\psi_N$ is resampled from a fixed measure $\pi_\psi(\psi_N)=\pi_{p}(\textbf{p},N)\pi_{p_n}(p_n)$ and thus
\begin{gather}
\Gamma_T(\varphi_N\to\varphi_M)=\int \pi_\psi(\psi_N)\Gamma_T(\varphi_N,\psi_N\to \varphi_M, \psi_M) d\psi_Nd\psi_M
\end{gather}

To establish the detail balance, the following is the most important one.
\begin{lemma}\label{lmm:detailbalance}
Assume that $\pi_{N,N+1}(\mathbf{q}_{\mathrm{new}}, \mathbf{p}_{\mathrm{new}}; \mathbf{q}, \mathbf{p})$ is even in $\mathbf{p}$ and $\mathbf{p}_{\mathrm{new}}$.
For evolution time $T<1$, the transition kernel of DHMC dynamics satisfies the detailed balance condition.    
\begin{gather}\label{eq:detailbalance}
\bar\pi(\varphi_N,\psi_N)\Gamma_T(\varphi_N,\psi_N\to \varphi_M, \psi_M) = \bar\pi(\varphi_M,-\psi_N)\Gamma_T(\varphi_M,-\psi_M\to \varphi_N, -\psi_M).
\end{gather}
\end{lemma}
We attach the proof of Lemma \ref{lmm:detailbalance} in Appendix \ref{app:lmm1proof}.

\begin{theorem}
\label{detail_balance_T}
    For any given time $T>0$, the transition kernel of ideal DHMC algorithm with fixed simulation time $\Gamma_T(\varphi_N\to\varphi_M)$ satisfies the detailed balance condition with respect to $\hat\pi(\varphi)$
$$
\hat\pi(\varphi_N)\Gamma_T(\varphi_N\to\varphi_M)=\hat\pi(\varphi_M)\Gamma_T(\varphi_M\to\varphi_N)
$$
 Consequently, the ideal DHMC Algorithm \ref{DHMC} with random simulation time satisfies the detailed balance condition with respect to the target distribution of $\hat\pi(\mathbf{q},n)$
\end{theorem}

\begin{proof}

\noindent\textbf{Step 1}: We first show the relationship \eqref{eq:detailbalance} in Lemma \ref{lmm:detailbalance} satisfies for all $T>0$. To achieve this, we just need to show the relationship satisfies for $2T$ if it satisfies for $T>0$.

In fact, applying the Chapman--Kolmogorov equation in \eqref{eq:ckexpansion}, Lemma \ref{lmm:detailbalance} and the evenness of $\psi$, one has
\[
\begin{aligned}
&\ \bar\pi(\varphi_{N},\psi_{N})\Gamma_{2T}(\varphi_N,\psi_N\to \varphi_M, \psi_M)\\
=&\sum_{K=0}^\infty \int \bar\pi(\varphi_{N},\psi_{N})\Gamma_{T}(\varphi_N,\psi_N\to \varphi_K, \psi_K)\Gamma_{T}(\varphi_K,\psi_K\to \varphi_M, \psi_M)d\varphi_K d\psi_K\\
=&\sum_{K=0}^\infty \int \bar\pi(\varphi_{K},\psi_{K})\Gamma_{T}(\varphi_K,-\psi_K\to \varphi_N, -\psi_N)\Gamma_{T}(\varphi_K,\psi_K\to \varphi_M, \psi_M)d\varphi_K d\psi_K\\
=&\bar\pi(\varphi_{M},-\psi_{M}) \sum_{K=0}^\infty \int  \Gamma_{T}(\varphi_K,-\psi_K\to \varphi_N, -\psi_N)\Gamma_{T}(\varphi_M,-\psi_M\to \varphi_K, -\psi_K) d\varphi_K d\psi_K\\
=&\bar\pi(\varphi_{M},-\psi_{M})\Gamma_{2T}(\varphi_M,-\psi_M\to \varphi_N, -\psi_N)
\end{aligned}
\]

\noindent\textbf{Step 2}:  Now that the detail balance relationship \eqref{eq:detailbalance} holds for any time period $T$, we now integrate $\psi_N,\psi_M$ to obtain the desired result. Using the fact that $\pi(\varphi_{N},\psi_{N})=\pi(\varphi_N)\pi(\psi_N)$, one has
\begin{multline}
\hat\pi(\varphi_N)\Gamma_{T}(\varphi_N\to \varphi_M)=\int \bar\pi(\varphi_{N},\psi_{N})\Gamma_{T}(\varphi_N,\psi_N\to \varphi_M, \psi_M) \d\psi_N \d\psi_M\\
=\int \bar\pi(\varphi_{M},-\psi_{M})\Gamma_{T}(\varphi_M,-\psi_M\to \varphi_N, -\psi_N) \d\psi_N \d\psi_M\\
=\int \bar\pi(\varphi_{M},\psi_{M})\Gamma_{T}(\varphi_M,\psi_M\to \varphi_N, \psi_N) \d\psi_N \d\psi_M
=\hat\pi(\varphi_{M})\Gamma_{T}(\varphi_M\to \varphi_N).
\end{multline}
In the second last step, we have performed a change of variables $\psi_M\to -\psi_M$, $\psi_N\to -\psi_N$.

\vskip 0.1 in

\noindent\textbf{Step 3}:  Consider the randomized simulation time $T\sim\mathrm{UNIF}[T_{\text{min}},T_{\text{max}}]$. The transition kernel of the ideal DHMC algorithm [\ref{DHMC}] can be written as 
\[
\Gamma(\varphi_N,\psi_N\to \varphi_K, \psi_K) = \frac{1}{T_\text{max}-T_\text{min}}\int_{T_\text{min}}^{T_\text{max}}\Gamma_{T}(\varphi_N,\psi_N\to \varphi_K, \psi_K)\d T ,
\]
where $\Gamma_{T}(\varphi_N,\psi_N\to \varphi_K, \psi_K)$ is the transition kernel of the ideal DHMC algorithm with fixed evolution time $T$. Combined with the above result, the proof is complete.
\end{proof}

\section{Discretization of DHMC and discussion}\label{section:discretization}

The continuous dynamics presented in Algorithm \ref{t-H dynamics} is not suitable for practical simulations.
In this section, we discretize of the DHMC dynamics and find the connection of DHMC to existing methods.

\subsection{Discretized DHMC}

\begin{algorithm}[ht]
    \caption{Discretization of DHMC dynamics}\label{discretization}
    \begin{algorithmic}
        \REQUIRE $\textbf{q},\textbf{p},n,p_n$ with $n\in [N, N+1]$ and a time step $\Delta t$. 
         \STATE  $\textbf{q} \gets \textbf{q}+\frac{\Delta t}{2}\nabla K(\textbf{p})$
        \STATE $\textbf{p}\gets \textbf{p}-\frac{\Delta t}{2}\nabla\mathcal{U}(\textbf{q})$  
        \STATE $\delta n = \frac{\Delta t}{m_n}\text{sign}(p_n)$
        \IF{$n+\delta n\geq N+1$ \textbf{or} $n+\delta n< N$}
            \STATE  Propose the $(\textbf{q},\textbf{p})$ as in Algorithm \ref{t-H dynamics}  (If multiple jumps have occured, one just samples the new coordinates consecutively).
           \STATE Compute $\delta\cH$.
            \IF{$K_n(p_n)\geq \delta\cH$}
                \STATE  accept the proposal
                \STATE  $p_n\gets p_n-\text{sign}(p_n)\cdot m_n\cdot \delta\cH$ 
                \STATE  $n\gets n+\delta n$
            \ELSE
                \STATE  reject the proposal
                \STATE  $p_n\gets -p_n$
            \ENDIF
        \ENDIF
        \STATE $\textbf{p}\gets \textbf{p}-\frac{\Delta t}{2}\nabla\mathcal{U}(\textbf{q})$
        \STATE $\textbf{q}\gets \textbf{q}+\frac{\Delta t}{2}\nabla K(\textbf{p})$
       \end{algorithmic}
\end{algorithm}

The DHMC system involves four parameters: \( (\mathbf{q}, \mathbf{p}, n, p_n) \). When the dimension indicator \( n \) does not hit an integer, the system evolves according to the Hamiltonian equations in Eq. \ref{hmc_1}. We will adopt a standard discretization for the Hamiltonian dynamics. Moreover,  since the dynamics of \( (\mathbf{q}, \mathbf{p}) \) and \( (n, p_n) \) are decoupled most of the time when a dimension jump does not occur, we choose to simulate the dynamics of \( (\mathbf{q}, \mathbf{p}) \) and \( (n, p_n) \)  alternatively following the idea of operator splitting.

The dynamics of \((n, p_n)\) is relatively simple: when $n$ does not hit an integer, $p_n$ is unchanged while $n$ keeps increasing or decreasing with a uniform speed. Hence, we only need to check whether jumps have happened or not after a time step has been consumed.
The main computational cost arises from simulating \( (\mathbf{q}, \mathbf{p}) \). We use the Störmer-Verlet scheme to simulate the dynamics of \( (\mathbf{q}, \mathbf{p}) \). To improve numerical accuracy, higher-order symplectic integrators may be employed for simulating the dynamics of \( (\mathbf{q}, \mathbf{p}) \).

When \( n \) hits an integer, we may need to add or remove parameters, compute the energy barrier, and determine whether the dimension jump is successful. As discussed in section \ref{subsec:effectiveHal}, we allow multiple jumps during one single time step. With all these ingredients, a single step of the algorithm is shown in Algorithm \ref{discretization}.

\subsection{DHMC as a special realization of RJMCMC}

The HMC can be considered a specialization of the Metropolis-Hasting algorithm where the proposal mechanism is explicitly given by the Hamiltonian dynamics. Here we clarify that our proposed DHMC for transdimensional sampling is also a specialization of the general reversible jump Markov chain Monte Carlo (RJMCMC) \cite{rjmcmc}.

For convenience, we discuss this in the context of Bayesian model selection. 
Typically, an RJMCMC sampler includes ``between-models moves" and ``within-model moves", enabling it to explore the whole parameter space.
When jumping from a model with lower dimension, say $d_N$, to a higher one, say $d_M$, one samples a random vector $u\sim \pi_{N\to M}(\cdot)$ of dimension $d_{N\to M}$ and then propose the new parameter $q^M = g_{N\to M}(q^N,u)|_1^{d_M}$ according to the one-to-one mapping $g_{N\to M}(q^N,u)$. Reversely, when jumping to a low dimension, we sample a random vector $v\sim \pi_{M\to N}(\cdot)$ of dimension $d_{M\to N}$ and then propose the new parameter $q^N=g_{M\to N}(q^M,v)|_1^{d_N}$, where $g_{M\to N}=g_{N\to M}^{-1}$ is chosen to be the inverse of $g_{N\to M}$ and $d_N+d_{N\to M}=d_M+d_{M\to N}$. As mentioned in \cite{rjmcmc}, the acceptance rate is
\begin{multline}\label{acc}
\mathrm{acc}(q^N,N\to q^M,M)=\\
\min\left\{1,\frac{\pi(q^M,M)\Gamma_n(M\to N)\pi_{M\to N}(v)}{\pi(q^N,N)\Gamma_n(N\to M)\pi_{N\to M}(u)}\bigg|\det\frac{\partial g_{N\to M}(q^N,u)}{\partial(q^N,u)}\bigg|\right\},
\end{multline}
which guarantees the detailed balance condition.

For illustration, we set $m_n=1$ and consider the DHMC starts at $n=N+1/2$ and $T\leq1$ in one step. As described in Algorithm \ref{DHMC} and \ref{t-H dynamics}, when the indicator $n$ moves within an interval $(N,N+1)$, DHMC performs only ``within-model moves". When the indicator $n$ goes across an integer $N$, a ``between-models move" is proposed, and it's accepted or rejected according to the magnitude of $p_n$. 

If $T< 1/2$, only ``within-model moves" are conducted, and DHMC reduces to HMC. If $T\geq 1/2$, the particle moves to the right and to the left with the same probability. 
The DHMC first makes a ``within-model move", then it proposes a ``between-models move", and finally makes another ``within-model move". The ``within-model moves" are done by evolving the Hamiltonian dynamics in the corresponding parameter space. Suppose it jumps from a ``model" with dimension $d_N$ to another with dimension $d_{N+1}$.
Since $p_n\sim \exp(-|p_n|)$ and does not change before $n$ hits an integer, the acceptance probability of this ``between-models move" proposal is
\begin{gather}\label{eq:accdis1}
\mathbb P(K_{n}(p_n)>\delta \cH|p_n\geq 0)=\int_{\{K_{n}(p_n)>\delta \cH\}}\pi_{p_n}(p_n|p_n\geq 0)\d p_n=\min\{1,e^{-\delta\cH}\}.
\end{gather}

Putting DHMC in the framework of the RJMCMC model, the dimension matching random vector for the forward move is $u =(p^N,q_{\mathrm{new}},p_{\mathrm{new}})$ and the mapping is defined as 
$$
g_{N\to N+1}(q^N,u)=F_{T_2}([F_{T_1}(q^N,p^N),q_{\mathrm{new}},p_{\mathrm{new}}])
$$
where $F_t$ is the solution map of the Hamiltonian dynamics for evolution time $t$ and $T_1=1/2,\ T_2=T-1/2$ are the evolution times before and after the "between-models move" happens. Note that $F_{T_1}$ is a flow map on a parameter space of dimension $2d_N$ and $F_{T_2}$ is a flow map on a parameter space of dimension $2d_{N+1}$. By the volume-preserving property of Hamiltonian dynamics, we have 
$$\det\frac{\partial g_{N\to N+1}(q^N,u)}{\partial(q^N,u)}=1.$$
For the reversed move, $v=-p^{N+1}$. Note that $-p^N\sim \mathcal N(0,\frac{m}{\beta}I_N), p^{N+1}\sim \mathcal N(0,\frac{m}{\beta}I_{N+1})$ and $(q_{\mathrm{new}},p_{\mathrm{new}})\sim\pi_{N,N+1}(\cdot,\cdot;q^N,p^N)$. According to equation \eqref{eq:gap1} and equation \eqref{eq:gap2}, the acceptance rate in \eqref{acc} should be 
\[
\min\left\{1,\frac{\pi(q^{N+1},N+1)e^{-\beta K(p^{N+1})}/ Z_p^{d_{N+1}}}{\pi(q^N,N) 
 \pi_{N,N+1}(q_{\mathrm{new}},p_{\mathrm{new}};q^N,p^N) e^{-\beta K(p^N)}/ Z_p^{d_N}}\right\},
\]
which yields the same result as in \eqref{eq:accdis1}. Therefore, the DHMC can be conceptually situated within the framework of RJMCMC.

\section{DHMC for grand canonical ensemble}\label{sec:grandcanonical}

In this section, we focus on the grand canonical ensemble for a physical system and focus on the detailed constructions of the DHMC method for such distributions. 

The grand canonical ensemble characterizes a particle system with a constant volume $V$, maintained at a constant chemical potential $\mu$ and a constant temperature $T$ through contact with a reservoir.  The grand canonical ensemble is sometimes referred to as the constant $\mu VT$ ensemble. The system can exchange both particles and energy with the reservoir.   Due to its constant $\mu VT$ nature, the grand canonical ensemble simulation is most preferred in the study of absorption phenomena.  The probability distribution for the grand canonical ensemble is given by
\begin{gather}\label{eq:gce}
\pi(q^N,N)=\frac{1}{Z_{\beta,\mu}}\frac{1}{N!}e^{\beta\mu N-\beta U(q^N)}\quad N\in\mathbb N_+, q^N\in\Omega^N\subset\mathbb (\R^{d})^N,
\end{gather}
where $Z_{\beta,\mu}=\sum_{N=0}^\infty \int_{\Omega^N} \frac{1}{N!} e^{\beta\mu N-\beta U(q^N)}dq^N$ is the partition function, $N$ is the number of particles, $\mu$ is the chemical potential, $\beta = \frac{1}{k_B T}$, $k_B$ is the Boltzmann constant and $U(q^N)$ is the potential of $N$ particles.

Various classical methods have been developed for the grand canonical ensemble. The traditional method is based on the Metropolis-Hastings algorithm scheme. Proposals for new configurations include random displacements, insertions, and removals of particles. The acceptance rate is then computed according to the different proposals. For more details, readers may refer to Chapter 5 of \cite{10.5555/547952}. Some references \cite{cagin1991grand,lynch1997grand,boinepalli2003grand} introduce a "fraction particle" to integrate the insertion/deletion process with the continuous dynamic. Temperature control is usually done by applying a thermostat, for example, the Langevin thermostat \cite{boinepalli2003grand} and the Nose Hoover thermostat \cite{shroll1999molecular,cagin1991grand,lynch1997grand}.  

Next, we show how our generalized DHMC framework presented above can be applied to sample from the grand canonical ensemble. 
Moreover, the method could be efficient if it is combined with the random batch strategy when we consider the interacting particle systems.

\subsection{DHMC for the Grand Canonical Ensemble}

Following the process in section \ref{subsec:effectiveHal}, we may extend the distribution in \eqref{eq:gce} to the following distribution
in extended space
\begin{gather}
\label{dist:gce}
\bar{\pi}(q,p,n,p_n) = \frac{1}{Z_{\beta,\mu,m_n}}\sum_{N\ge 0}\frac{1}{N!}e^{\beta\tilde\mu N-\beta (U(q^N)+K(p^N)+\frac{|p_n|}{m_n})}\mathbf{1}_{\{N\leq n<N+1\}}
\end{gather}
where
\begin{gather}
\tilde\mu=\mu-\frac{1}{\beta}\log Z_p=\mu-\frac{d}{2\beta}\log\left(\frac{2m\pi}{\beta}\right)
\end{gather}
is the modified chemical potential when we consider the momentum, with $Z_p=(\frac{2m\pi}{\beta})^{d/2}$. Recall that the $\pi$ here is the constant. Moreover, the normalizing constant is given by
\begin{gather}
Z_{\beta,\mu,m_n} = Z_{\beta,\mu}\cdot\frac{2m_n}{\beta}.
\end{gather}

We now accommodate the DHMC method in section \ref{subsec:dhmc} to the grand canonical ensemble. We denote that 
\begin{gather}
    \cU_{N}(q^N)=U(q^N)+\frac{1}{\beta}\log(N!)-\tilde{\mu}N   
\end{gather}
and \eqref{dist:gce} becomes
\begin{gather}
    \bar{\pi}(q,p,n,p_n) = \frac{1}{Z_{\beta,\mu,m_n}}\sum_{N\ge 0}e^{-\beta (\cU_N(q^N)+K(p^N)+\frac{|p_n|}{m_n})}\mathbf{1}_{\{N\leq n<N+1\}}
\end{gather}
Since $q^N=(q_1, \cdots, q_N)\in (\R^{d})^N$ and $p^N=(p_1,\cdots, p_N)\in (\R^d)^N$, when $N$ is increased by $1$, we just need to add
one particle. For our setting, it would be convenient to choose
\begin{gather}
q_{N+1}\sim \mathrm{UNIF}(\Omega),\quad p_{N+1}\sim \mathcal{N}\left(0,\frac{m}{\beta}I_d\right).
\end{gather}
Hence, the density for the new particle  is simply 
\[
\pi_{N,N+1}=\frac{1}{VZ_p}\mathbf{1}_{\Omega}(q_{N+1})\exp\left(-\beta\frac{|p_{N+1}|^2}{2m}\right).
\]
where $V=|\Omega|$. With this choice, we find easily that
 \begin{multline}
    \label{eq:gcadd}
    \delta \cH = \cU_N(q^{N+1})+K(q^{N+1})-\cU_N(q^N)-K(q^{N}) + \frac{1}{\beta}\log(\pi_{N,N+1}(q_{N+1},p_{N+1};q^N,p^N))\\
    =U(q^{N+1})-U(q^N) +\frac{1}{\beta}\log (N+1)-\frac{1}{\beta}\log(V) -\mu.
 \end{multline}
Note that $\mu$ here is the original chemical potential in \eqref{eq:gce}.

When we delete one particle, the free energy barrier is similarly given by
\begin{equation}\label{eq:gcdelete}
\delta \cH = U(q^{N-1})-U(q^{N})-\frac{1}{\beta}\log N+\frac{1}{\beta}\log(V)+\mu.
\end{equation}

Note that according to the method in section \ref{subsec:dhmc}, we may delete
the last particle in the list. This is not very beneficial because the particles play symmetric roles.
Hence, we will instead choose a random particle from $i=1,\cdots, N$ to delete. That means 
\[
q^{N-1}=(q_1,\cdots, q_{\xi-1}, q_{\xi+1},\cdots, q_N ),
\]
with $\xi$ uniformly sampled from $1,\cdots, N$.
This extra random process will not bring in a new correction in the energy barrier.
In fact, the desired grand canonical ensemble is symmetric in the particles with respect to any permutation and we will then consider statistics that are invariant under permutation as well. Hence, we can understand the samples the same if they are different by a permutation.  With this understanding, our DHMC method can be imagined as the algorithm in the state space by modulo the permutation symmetry.


To summarize, for each iteration of the DHMC, the momenta are resampled by 
\[
\bm{p}^{N}\sim\mathcal{N}(0,I_{dN}),\quad p_n \sim  \frac{\beta}{2m_n}e^{-\beta|\frac{p_n}{m_n}|}.
\]
During the simulation, When a particle is added, we do the sampling $q_{N+1}\sim \mathrm{UNIF}(\Omega)$ ,$p_{N+1}\sim \mathcal{N}(0,\frac{m}{\beta}I_d)$. When a particle is deleted, we choose randomly one to delete. When a particle is added or deleted, we compute $\delta\cH$ given by \eqref{eq:gcadd} or \eqref{eq:gcdelete}.
Then, Algorithm \ref{discretization} is easily modified for grand canonical ensemble.

\begin{remark}
    While the previous discussion considered \(\Omega\) to be bounded, we now turn to the full space case where \(\Omega = \mathbb{R}^d\). We just need to adjust the distribution of the new particle \(\pi_{N,N+1}\), while \(\delta \cH\) is still computed using Equation \eqref{eq:gap1}. Specifically, we may choose  
$$
q_{N+1} \sim \pi_\mathrm{new}(q)= \frac{1}{Z_\beta}e^{-\beta V(q)}, \quad p_{N+1} \sim \mathcal{N} \left(0, \frac{m}{\beta} I_d \right).
$$
where $Z_\beta=\int e^{-\beta V(q)}\mathrm{d} q$. With this choice, the energy barrier for adding a particle is given by  
$$
\delta \cH = U(q^{N+1}) - U(q^N) + \frac{1}{\beta} \log (N+1) - \frac{1}{\beta} \log(Z_\beta) - V(q_{N+1}) - \mu.
$$
Similarly, the energy barrier for deleting a particle is given by  
$$
\delta \cH = U(q^{N-1}) - U(q^N) - \frac{1}{\beta} \log N + \frac{1}{\beta} \log(Z_\beta) + V(q_N) + \mu.
$$
\end{remark}

\subsection{Combining DHMC with Random Batch Method}

In many cases, we need to sample from a distribution $\pi(q^N, N) = \frac{1}{Z} \frac{1}{N!} e^{-U(q^N)}$, where the potential function is of the form 
\[
U(q^N) = \sum_{i=1}^N V(q_i)+\sum_{1\leq i<j\leq N} \phi(q_i-q_j).
\]
Here, $V(q_i)$ is the confining potential and $\phi(q_i-q_j)=\phi(q_j-q_i)$ is the interaction potential between particle $q_i$ and $q_j$. When simulating Hamiltonian dynamics as described in section \ref{subsec:dhmc}, the gradient of the potential $\mathcal{U}$ needs to be calculated and updated at each time step, and the computational cost is $O(N^2)$ due to the interaction part of the potential, which becomes computationally expensive for large $N$.
$$
\nabla_{q_i}U = \nabla V(q_i)+\sum_{j\neq i}\nabla  \phi(q_i-q_j),\quad i=1,2,\ldots,N 
$$
To mitigate this computational burden, we employ the well-known random batch method (RBM). Originally introduced in \cite{Jin_2020} for simulating the interacting particles, RBM is conceptually simple: at each time step, the particles are randomly divided into smaller groups, and interactions are computed only within each group (batch). This approach provides an unbiased approximation of the full gradient while significantly reducing the computational cost to $O(N)$. RBM has also been extended to second-order Langevin dynamics, as proposed in \cite{RBM_2ndorder}. As shown in \cite{Jin_2020, RBM_2ndorder}, under mild conditions, RBM achieves uniform-in-time convergence for both first-order and second-order Langevin dynamics, with a convergence rate that is independent of the number of particles. This is due to time averaging, which can be interpreted as a law of large numbers in time.  
RBM has also been applied to sampling problems, as discussed in \cite{li2023splittinghamiltonianmontecarlo, li2020random}, where it is used to speed up the Hamiltonian Monte Carlo and Langevin Monte Carlo for particle systems.  

When the interaction kernel \( \nabla \phi \) exhibits singularity, the vanilla RBM may not perform well. To address this issue, kernel splitting was proposed in \cite{li2023splittinghamiltonianmontecarlo, li2020random, RBM_2ndorder} to improve RBM’s effectiveness. The key idea is to decompose the interaction kernel into two parts:  
\[
\nabla \phi(z) = K_1(z) + K_2(z),
\]  
where \( K_1(z) \) contains the singularity but has short-range interactions, while \( K_2(z) \) is smooth and bounded, potentially with long-range interactions. In this paper, we adopt the approach in \cite{RBM_2ndorder}, and apply RBM only to the smooth component \( K_2 \), while computing the full interaction of \( K_1 \) directly, as detailed in Algorithm \ref{RBM_singular}. Notably, if \( \nabla \phi(z) \) is non-singular, we can simply set \( K_1(z) = 0 \) and \( K_2(z) = \nabla \phi(z) \).  The random batch approximation of $\nabla U$ will be used in Algorithm \ref{discretization} (the same approximation is used for the two half-steps during a single loop) for the Lennard-Jones example in section \ref{subsec:LJ}.

\begin{algorithm}[ht]
    \caption{RBM}\label{RBM_singular}
    \begin{algorithmic}
        \REQUIRE $q^{N}$
        \STATE Divide $\{1,2,\cdots,N\}$ into n batches randomly.
        \FOR{each batch $C_{\ell}$}
            \FOR{$i$ in $C_{\ell}$}
                \STATE $$\nabla_{q_i}U \leftarrow \nabla V(q_i)+\sum_{j\neq i}K_1(q_j-q_i) + \frac{N}{|C_{\ell}|}\sum_{j \in C_{\ell}\ \text{and}\ j\neq  i} K_2(q_j-q_i)$$
            \ENDFOR
        \ENDFOR
    \end{algorithmic}
\end{algorithm}

When sampling from grand canonical ensemble for interacting particles using Algorithm \ref{discretization} combined with the random batch approximation in Algorithm \ref{RBM_singular}, the computational complexity is then reduced from $O(N^2)$ to $O(N)$. In fact, for each iteration, each Random Batch gradient computation takes $O(N)$. Each time we add or delete a particle, the computation of energy gap takes $O(N)$, so the complexity of RB-DHMC is $O(N)$ per iteration.

\section{Numerical Examples}\label{sec:numerical}
In this section, we present three numerical examples to validate the correctness of our theory and demonstrate the efficiency of our algorithm. All the systems in this section are placed in a fixed periodic box.
The first example is the 1-dimensional free gas model and the second example is a 1-dimensional artificial system of interacting particles with an artificial cosine interaction function, which are used to test the correctness of the method. 
The third example is the Lennard-Jones fluid using which our DHMC algorithm combined with the random batch strategy is tested, showcasing both the accuracy and the computational efficiency of this approach when compared to the standard Metropolis-Hastings (MH) algorithm.

\subsection{A one dimensional free gas model}\label{subsec:free_gas}
In this one dimensional free gas model, the particles do not have interaction.  Thus, the potential $U(q^N)\equiv 0$. We assume that the whole system is placed in a periodic box with length $L$. The distribution of the number of particles in the system is as follows.

$$
\pi(q^N) = \frac{e^{-Le^{\beta\mu}}}{N!}e^{\beta \mu N}, \quad q^N\in [0,L]^N,N\in \mathbb N
$$
hence the probability mass function of the number of particles $N$ in the system is
\begin{equation*}
    \pi_N(k)=\frac{e^{-Le^{\beta\mu}}}{k!}e^{\beta \mu k}L^k\quad k\in\mathbb N
\end{equation*}
so the number of particles in the system follows the Poisson distribution with parameter $\lambda=Le^{\beta\mu}$

The  total variation distance between the theoretical particle number distribution and the sample particle number distribution is defined as follows, 
\[
TV(\pi_N,\tilde{\pi}_N):=\sum_k|\pi_N(k)-\tilde\pi_N(k)|,
\]
where $\pi_N(k)$ and $\tilde\pi_N(k)$ are the theoretical particle number distribution and sample particle number distribution, respectively.

\begin{figure}
    \centering
    \includegraphics[width=0.9\linewidth]{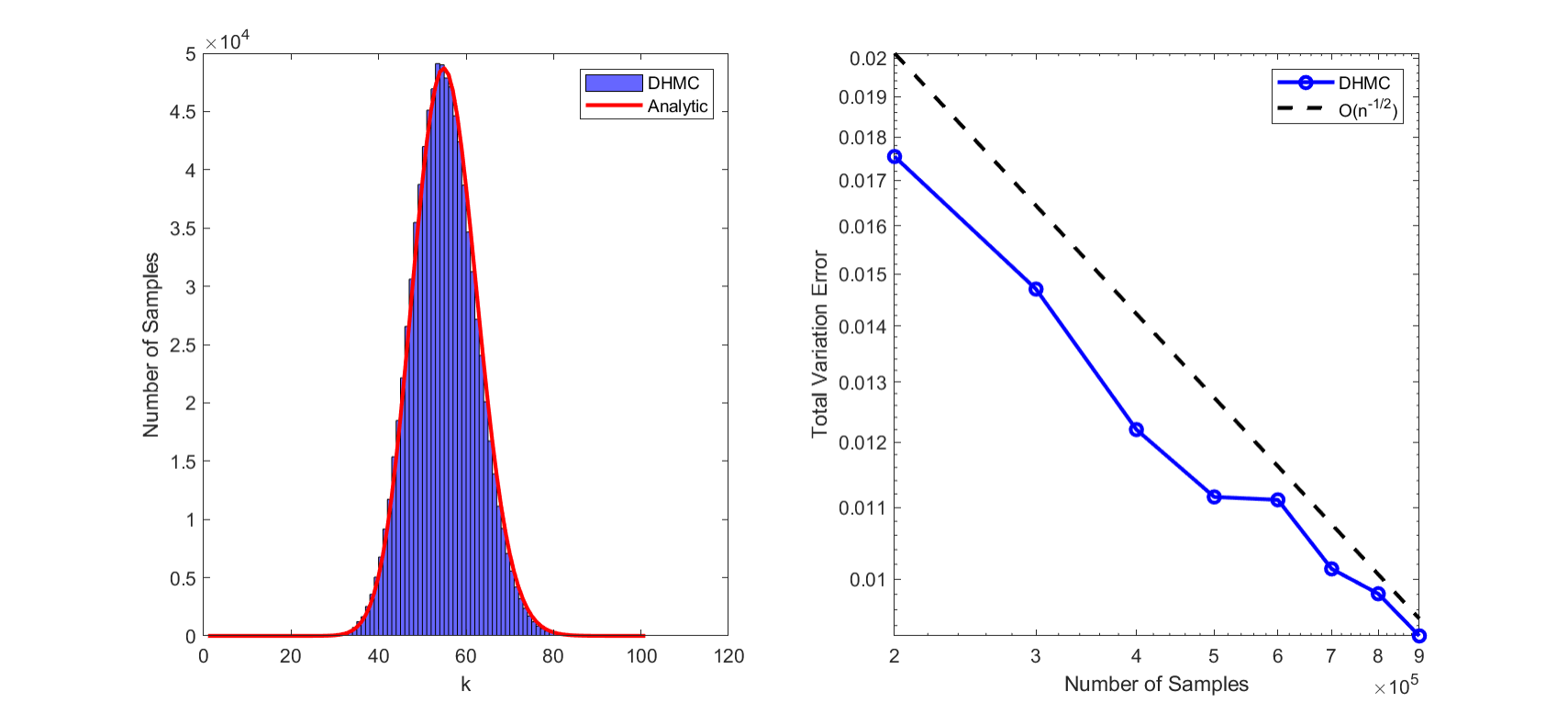}
    \caption{Left panel: Empirical distribution of particle numbers for DHMC samplers with a sample size of $9\times10^5$; Right panel: Error versus different sample sizes in log scale, with a reference line having a slope of -1/2}
    \label{fig:free_gas}
\end{figure}

In this numerical example, the step size of the integrator is $\epsilon\sim \mathrm{UNIF}[0.05,0.1]$ and the number of numerical integrations for every new sample is $N_\text{iter}=5$. We use a periodic boundary.

When we add a new particle into the system, the position is uniformly sampled in $[0,L]$ and the momentum is sampled from the Gibbs distribution with respect to the kinetic energy $K(p)$ at temperature $\beta^{-1}$. In this case, the dimension-jumping free energy gap is 
\[
\delta\cH(q^N,p^N\to q^M,p^M) = \frac{1}{\beta}(\log(M!)-\log(N!)) + (M-N)(\frac{\log(L)}{\beta} - \mu ).
\]

As shown in Figure \ref{fig:free_gas}, the particle number distribution from the DHMC algorithm matches the theoretical distribution.  Moreover, the convergence rate with respect to the sample size decays as $O(N^{-1/2})$, agreeing with the standard Monte Carlo rate, showing that our method can correctly sample from the grand canonical ensemble if there is no interacting potential.




\subsection{A system with an artificial cosine potential}

In this example, we consider the artificial 1-dimensional model with cos-interactive-potential.
\[
U(q^N) = \sum_{i<j}\text{cos}(\frac{2\pi}{L}(q_i-q_j)).
\]
It is hard to explicitly compute the distribution of the number of particles. Thus we compare the result of DHMC with the Metropolis-Hastings algorithm (in particular RJMCMC).


For our extended DHMC method, we adopt the same method of adding a new particle as in the previous example, i.e. $q_{\mathrm{new}}\sim \mathrm{UNIF}([0,L])$ and $p_{\mathrm{new}}\sim\mathcal{N}(0,\frac{m}{\beta})$. So the energy barrier of dimension jump is as follows
$$
\delta \mathcal{H}(q^N,p^N\to q^M,p^M) = \frac{1}{\beta}(\log(M!)-\log(N!)) + (M-N)(\frac{\log(L)}{\beta} - \mu ) + U(q^M)-U(q^N).
$$
In the simulation, we set \( L = 10 \), \( \beta = 1 \), \( m = 1 \), \( m_n = 1 \), and \( \mu = -0.5 \). Similarly, we adopt the periodic boundary condition. For the DHMC simulation, in each iteration, the number of numerical integration is set to be $N_{\text{int}}=5$, and the step size is set to be  $\epsilon\sim\mathrm{UNIF}([0.05, 0.1])$. To make sure we are sampling from the equilibrium, we take the burn-in steps to be $10^6$.



To demonstrate the convergence of the weak error, we consider two metrics: the total variation distance of the marginal distribution with respect to the particle number (defined in the same way as in Section~\ref{subsec:free_gas}) and the weak error on a test function that depends on both $q^N$ and $N$. The test function $\phi(q^N, N)$ is defined as
\[
\phi(q^N, N) = \sum_{i<j} \cos^2\left(\frac{2\pi N}{L}(q_i - q_j)\right).
\]

The weak error for this test function is defined by
\[
\text{err}_{\text{weak}} = \frac{1}{N_{\text{sample}}} \sum_{i=1}^{N_{\text{sample}}} \left| \phi(\mathbf{q}^{(i)}_{\text{DHMC}}, N^{(i)}_{\text{DHMC}}) - \overline{\phi(\mathbf{q}_{\text{MH}}, N_{\text{MH}})} \right|,
\]
where $(\mathbf{q}^{(i)}_{\text{DHMC}}, N^{(i)}_{\text{DHMC}})$ denotes the $i$-th sample from the DHMC sampler. The reference value is given by the sample average from the Metropolis-Hastings (MH) sampler:
\[
\overline{\phi(\mathbf{q}_{\text{MH}}, N_{\text{MH}})} = \frac{1}{N'_{\text{sample}}} \sum_{j=1}^{N'_{\text{sample}}} \phi(\mathbf{q}^{(j)}_{\text{MH}}, N^{(j)}_{\text{MH}}).
\]
Similar to Chapter 6 of \cite{10.5555/547952}, the proposal step of the MH sampler is generated by (1)
With probability $0.4$, add a new particle$  q_\mathrm{new} \sim \mathrm{UNIF}([0, L]) $;
(2) With probability $0.4$, randomly select a particle and then remove it; (3) With probability $0.2$, randomly select  $\lfloor N/5 \rfloor$  particles, and assign each selected particle a new independent position drawn from  $\mathrm{UNIF}([0, L])$. To ensure the MH samples are drawn from equilibrium, we use a burn-in of $10^6$ steps and a sample size of $N'_{\text{sample}} = 7.6 \times 10^8$, which provides a highly accurate estimate of the true expectation of the test function $\phi$.



\begin{figure}[H]
    \centering
    \includegraphics[width=\textwidth]{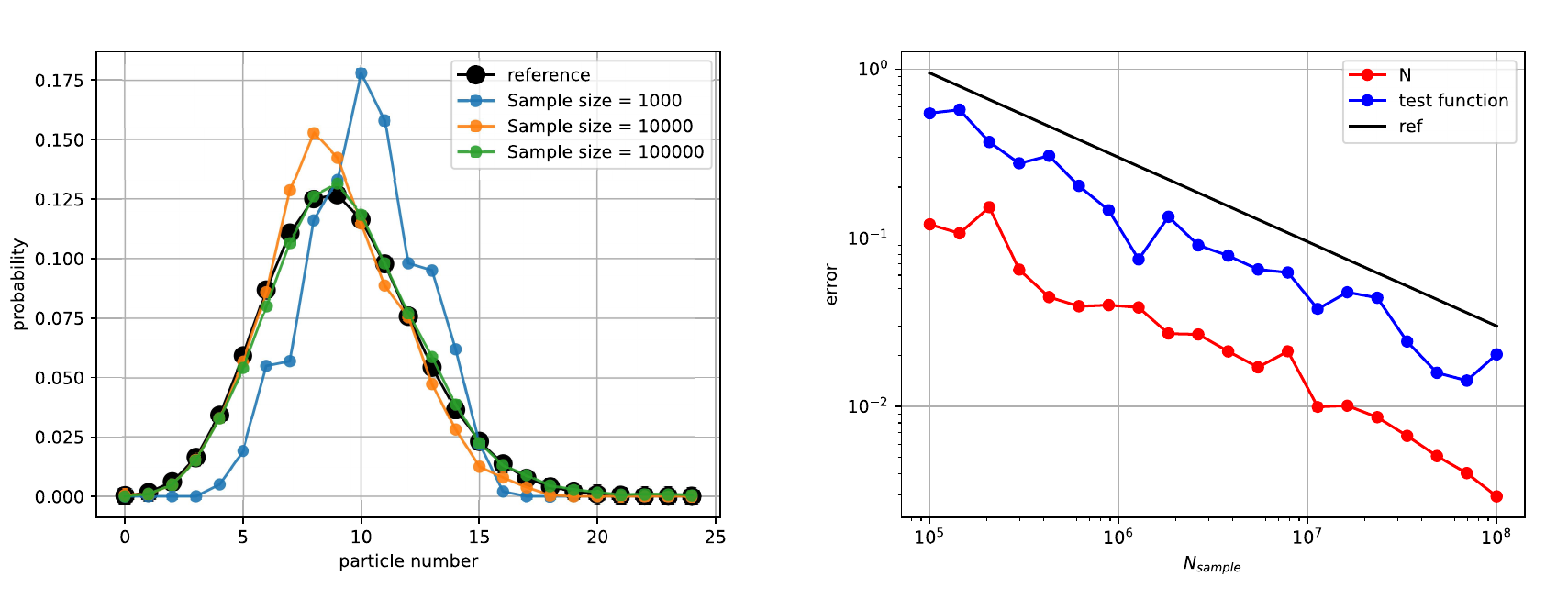}
    \caption{Left panel: Empirical distributions of particle numbers obtained using DHMC samplers with varying sample sizes. The reference distribution is given by the empirical distribution of the MH sampler with a sample size of $10^8$ at equilibrium.  Right panel: Error versus different sample sizes, with a reference line having a slope of -1/2}
    \label{cos}
\end{figure}


As shown in Figure \ref{cos}, the empirical distribution of particle number produced by DHMC matches the reference, produced from MH using fine steps and large sample size. As shown in the right panel, the weak error converge in half-order, verifying the correctness of our method. Note that for the curves in the right panel of Figure \ref{cos}, we repeated the numerical experiments $10$ times and reported the average error to reduce the variance of the plot. This averaging mitigates fluctuations due to randomness in individual runs.

\subsection{The Lennard-Jones fluid}\label{subsec:LJ}
In the final example, we apply our algorithm to the grand canonical Monte Carlo simulation of the Lennard-Jones fluid. We sample from the grand canonical distribution of a three-dimensional interacting particle system with Lennard-Jones potential 
\begin{equation*}
    U(q^N)=\sum_{i<j}\phi(r_{ij}), \quad \phi(r_{ij})=4\bigg( \frac{1}{|r_{ij}|^{12}}-\frac{1}{|r_{ij}|^6}\bigg), \quad r_{ij}=(q_i-q_j),
\end{equation*}
where $q_i$'s are the location of particles and $|r_{ij}|$'s are the distance between each two different particles. The simulation is conducted in periodic boxes, which means that a particle interacts with both other particles and their images. Hence the interacting potential now reads
\begin{equation*}
    U^*(q^N)=\sum_{n}^*\sum_{i<j}\phi(r_{ij}+nL),
\end{equation*}
where the notation $\sum_{n}^*$ means that $i=j$ is excluded from the latter summation when $n=0$. 

Due to the singularity and rapid decay of Lennard-Jones potential, we introduce cutoff distance $r_c$ and a splitting strategy for the random batch method. This splitting strategy is exactly the same as it is in \cite{li2023splittinghamiltonianmontecarlo}.
\begin{equation*}
    \phi(r)=\phi_1(r)+\phi_2(r),
\end{equation*}
where
\begin{equation*}
    \phi_1(r)=
    \begin{cases}
        -2^{-\frac{1}{6}}r,&0<r<2^{\frac{1}{6}},\\
        4(r^{-12}-r^{-6}),&r\geq2^{\frac{1}{6}},\\
    \end{cases}
\end{equation*}
and
\begin{equation*}
        \phi_2(r)=
    \begin{cases}
        4(r^{-12}-r^{-6})+2^{-\frac{1}{6}}r,&0<r<2^{\frac{1}{6}},\\
        0,&r\geq2^{\frac{1}{6}}.\\
    \end{cases}
\end{equation*}

With cutoff distance $r_c$, the pressure is approximated by
$$
P=\frac{\rho}{\beta}+\frac{8}{V}\sum_{i=1}^{N}\sum_{\substack{j:j>i,\\r^*_{ij}<r_c}}\bigg[2\big(\frac{1}{r^*_{ij}}\big)^{12}-\big(\frac{1}{r^*_{ij}}\big)^{6}\bigg]+\frac{16}{3}\pi\rho^2\bigg[\frac{2}{3}\big(\frac{1}{r_c}\big)^9-\big(\frac{1}{r_c}\big)^3\bigg],
$$
where $\rho=\frac{N}{L^3}$ is the number density of the system.

\begin{figure}[H]
    \centering
    \includegraphics[width=\textwidth]{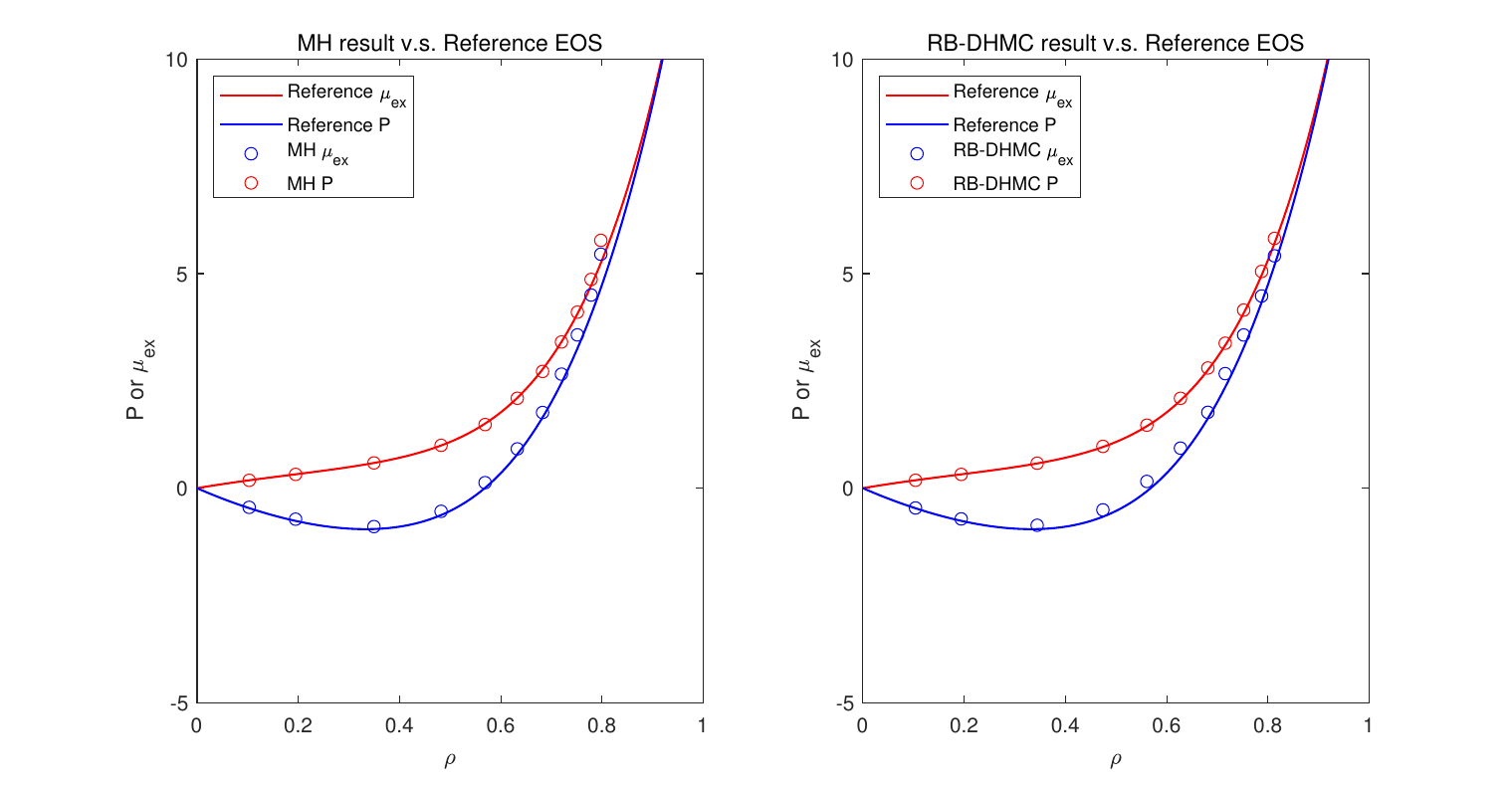}
    \caption{The line is the equation of state proposed by Johnson\cite{johnson1993lennard} and the marks are simulation results given by Metropolis-Hastings algorithm and our RB-DHMC.}
    \label{eos_lj}
\end{figure}

In this example, we adopt the kernel splitting and random batch method as in Algorithm \ref{RBM_singular}. We always take the batch size as $2$.
We impose $\mu\in\{-5,-4,...4,5\}$ and set $L=12.6$, $T=0.5$. The step size of a single step of the integrator $\Delta t$ is set to be $\Delta t\sim \mathrm{UNIF}[0.001,0.003]$ due to the singularity of the Lenord-Jones potential and the system is evolved for $\Delta t_n=5\Delta t$ to propose a new state.  The first $2\times 10^5$ iteration is discarded as the burn-in phase. The pressure is evaluated in the following way: we run the iteration for $M$ times and then compute the pressure for the last configuration, here the pressure is evaluated every $100$ iteration. As is shown in Figure \ref{eos_lj}, our simulation presents reasonable agreement with the equation of state proposed by \cite{johnson1993lennard}. Here, the red and blue lines are the fitted curves in \cite{johnson1993lennard}, while the red and blue marks (circles) are results by numerical simulations.

\begin{figure}[H]
    \centering
    \includegraphics[width=0.6\textwidth]{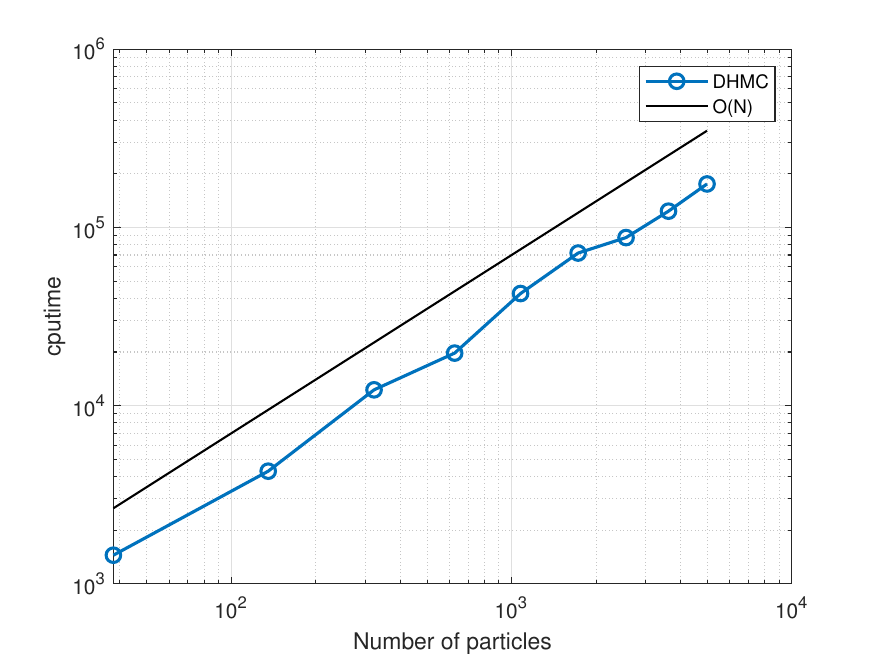} 
    \caption{The CPU time of RB-DHMC grows linearly with respect to the number of particles.}
    \label{time_lj}
\end{figure}
Note that our DHMC method combined with random batch method takes $O(N)$ per loop. Figure \ref{time_lj}  demonstrates the computational efficiency of RB-DHMC.
Clearly, the computational time scales linearly with the particle number, agreeing with the analysis above.

The MH method also takes $O(N)$ per time step as it only moves one particle. The complexity of the two methods are the same but our method consumes more time per time step as it evolves the whole particle system. Since our method moves $N$ particles, it should have smaller correlation.  As shown in Figure \ref{p_sample}, the pressure computed by DHMC clearly has less correlation and thus better quality.

\begin{figure}[H]
    \centering
    \includegraphics[width=0.8\textwidth]{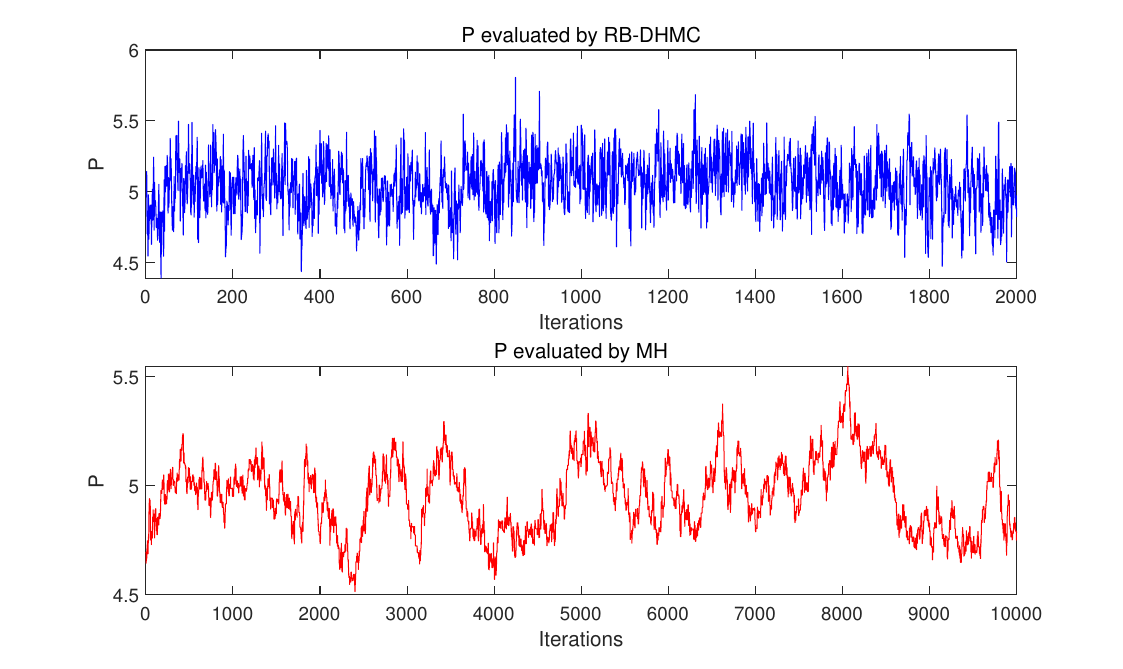}
    \caption{Upper: The pressure evaluated by RB-DHMC. Lower: The pressure evaluated by MH.}
    \label{p_sample}
\end{figure}

\begin{figure}[H]
    \centering
    \includegraphics[width=\textwidth]{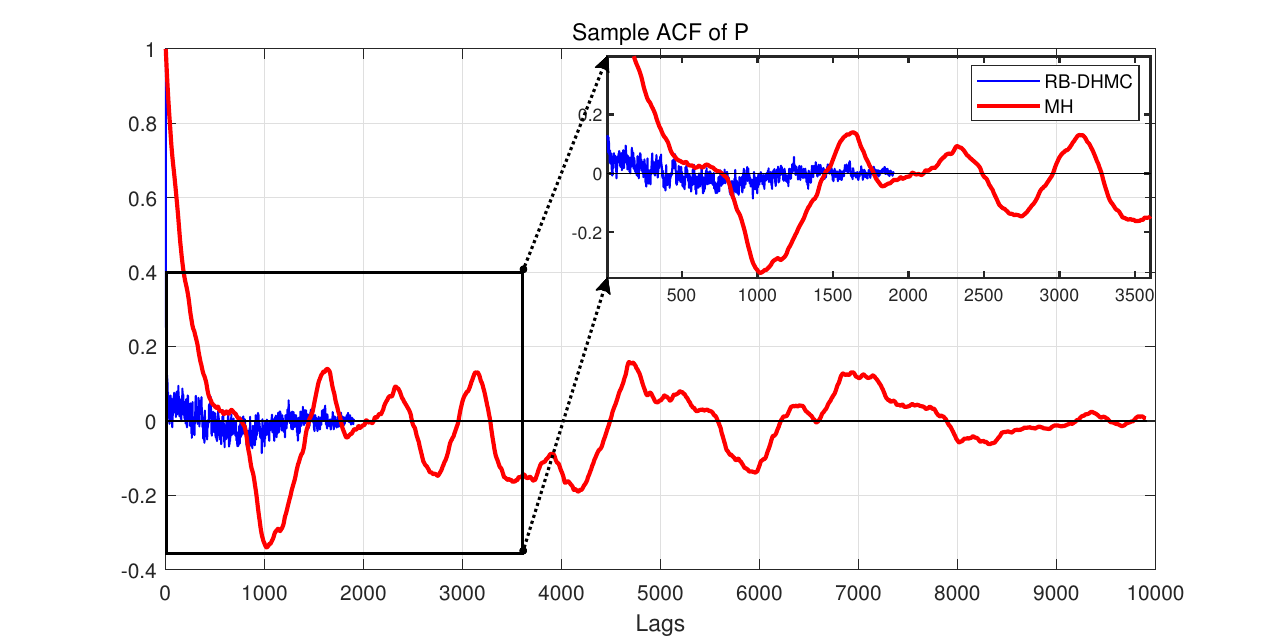}
    \caption{The sample ACF of P.}
    \label{acf_lj}
\end{figure}

Next, we demonstrate that our method can generate samples with much less correlation more quantitatively. The correlation is evaluated by the sample autocorrelation function(ACF). We compare our algorithm with the Metropolis-Hastings algorithm in terms of pressure evaluation. As shown in Figure \ref{acf_lj}, the ACF of RB-DHMC oscillates around $0$ with a significantly smaller amplitude, which implies better mixing and faster convergence.

\section*{Acknowledgement}

The work of L. Li was partially supported by the National Key R\&D Program of China, Project Number 2021YFA1002800,  NSFC 12371400 and 12031013, and  Shanghai Municipal Science and Technology Major Project 2021SHZDZX0102.

\bibliographystyle{plain} 
\bibliography{dhmc} 

\appendix

\section{Proof of Lemma \ref{lmm:detailbalance}}\label{app:lmm1proof}

We need to show the following detailed balance condition 
\begin{multline}
\label{detail}
\bar\pi(\bfq^N,\bfp^N,n,p_n)\Gamma_T(\bfq^N,\bfp^N,n,p_n\to \tilde \bfq^M,\tilde \bfp^M,\tilde n,\tilde p_n)\\
    =\bar\pi(\tilde \bfq^M,-\tilde \bfp^M,\tilde n,-\tilde p_n)\Gamma_T(\tilde \bfq^M,-\tilde \bfp^M,\tilde n,-\tilde p_n\to \bfq^N,-\bfp^N,n,-p_n)
\end{multline}
Recall that we have assumed $m_n = 1$ and $T<1$. 

First, let us recall some well-known properties  of the Hamiltonian dynamics. Let $F_T$ denote the solution map of the following Hamiltonian dynamics for evolution time $T$
\begin{equation}
\label{hamilton}
    \begin{dcases}
        \frac{d\mathbf{q}^N}{dt} = \nabla_{\mathbf{p}^N} \mathcal{H}_N, \\
        \frac{d\mathbf{p}^N}{dt} = -\nabla_{\mathbf{q}^N} \mathcal{H}_N,
    \end{dcases}
\end{equation}
where $\mathcal{H}_N = -\frac{1}{\beta} \log \pi(\mathbf{q}^N, \mathbf{p}^N)$.
Then
\begin{itemize}
    \item[(1)] $F_T$ is a diffeomorphism with the determinant of the Jacobian of $F_T$ being $1$, i.e.,  $\det J{F_T} \equiv 1$.
    \item[(2)] $F_T$ preserves the value of the Hamiltonian $\mathcal{H}_N$ along the trajectory exactly.
    \item[(3)] $F_T$ is reversible, i.e., if $F_T(\mathbf{q}^N, \mathbf{p}^N) = (\tilde{\mathbf{q}}^N, \tilde{\mathbf{p}}^N)$, then $F_T(\tilde{\mathbf{q}}^N, -\tilde{\mathbf{p}}^N) = (\mathbf{q}^N, -\mathbf{p}^N)$.
\end{itemize}

\begin{figure}[H]
    \centering
    \includegraphics[width=0.8\textwidth]{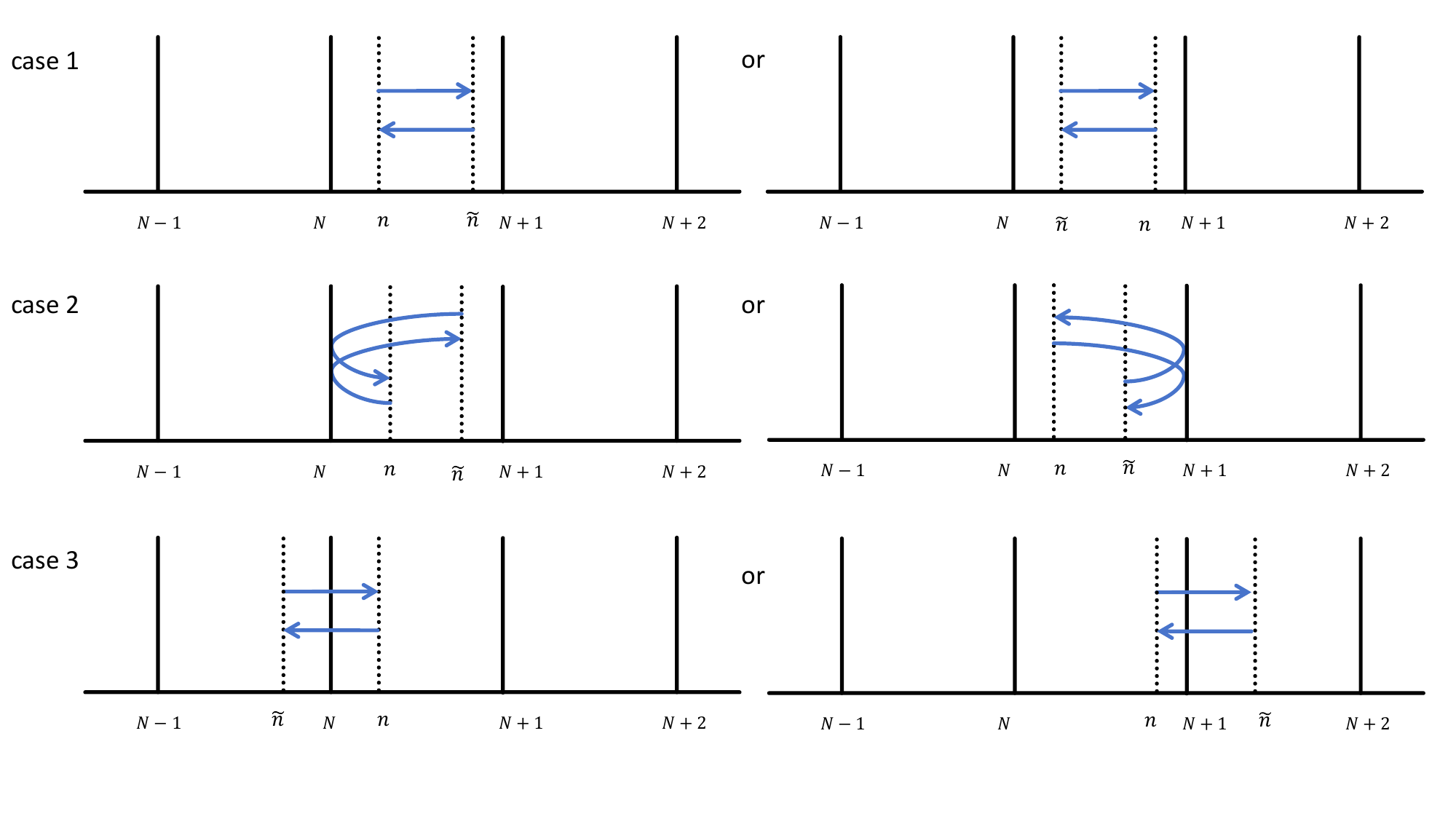} 
    \caption{Among all possible cases of the motion of $n$, only those shown in the figure need to be considered when $\Delta n <1$}
    \label{n_cases}
\end{figure}

We can now turn to the study of $\Gamma_T(\bfq^N,\bfp^N,n,p_n\to \tilde \bfq^M,\tilde \bfp^M,\tilde n,\tilde p_n)$. Note that for given $(\bfq^N, \bfp^N, n, p_n)$, the evolution is deterministic unless $n$ hits an integer.
Similarly, for a given $(\tilde \bfq^M,\tilde \bfp^M,\tilde n,\tilde p_n)$, if we evolve it backward in time, the dynamics is also deterministic until it hits one integer. Hence, one could easily determine whether they could be connected by our DHMC algorithm. Since  $\Delta n = T<1$, there are three cases as shown in Figure \ref{n_cases}.

\vskip 0.1 in

\noindent
\textbf{Case 1}: $p_n\geq0$ and $T<N+1-n$ or $p_n<0$ and $T< n-N$ 

The indicator $n$ does not come across any integer and thus the nontrivial contribution comes from $M=N$. Hence, the transition kernel of DHMC degenerates into the transition kernel of HMC, and the detailed balance condition (\ref{detail}) holds.

\vskip 0.1 in

\noindent
\textbf{Case 2}: $p_n<0$ and $T\ge n-N$ or $p_n\ge 0$ and $T\geq N+1-n$, but $p_n$ is not large enough to cross the boundary. In this case, $M=N$. (Note that when $p_n> 0$, the transition kernel would include $M=N$ (discussed here) and $M=N+1$ (discussed in case 3). Similar situation happens for $p_n<0$.)

When $p_n\ge 0$, the indicator $n$ increases and then hits an integer. Then, insertion is proposed. The probability that $K(p_n)<\delta \cH$ is given by
\[
P=\iint_{\Omega^{N-1,N}\times \R^{d_{N}-d_{N-1}}}\mathbf{1}_{\{K(p_n)<\delta \cH\}}\pi_{N,N+1}(\bfq_{new},\bfp_{new};F_{T_1}(\bfq,\bfp))d \bfq_{new} d\bfp_{new}
\]
where $T_1=N-n$. Here, $\Omega^{N-1,N}$ is the domain for $\bfq_{new}$ with the given $F_{T_1}(\bfq,\bfp)$.
In this case, the indicator $n$ is bounced back by negating $p_n$. 
When $p_n<0$, whether $K(p_n)<\delta \cH$ or not is deterministic. Hence, $P=0$ or $P=1$. 

Note that as long as the proposal dimension change is rejected, the dynamics of $(\textbf{q},\textbf{p})$ later would not be altered so we only care about the probability of rejection in this case.
Then, we have
\begin{multline*}
    \Gamma_T(\bfq^N,\bfp^N,n,p_n\to \tilde \bfq^{N},\tilde \bfp^{N},\tilde n,\tilde p_n)=\\
    P\delta ((\tilde \bfq^N,\tilde \bfp^N)-F_{T}(\bfq^N,\bfp^N))\delta(n+\tilde n-2N-\Delta n)\delta(\tilde p_n+p_n).
\end{multline*}
If we reverse the dynamics, starting with $(\tilde \bfq^N, -\tilde \bfp^N, \tilde{n}, -\tilde{p}_n)$ that makes the transition kernel above nonzero, the system will hit the same integer as the Hamiltonian dynamics is reversible. Then, the probability of rejection $P$ will be the same as by the conservation of Hamiltonian and the evenness in $p$. By the volume-preserving, the detailed balance condition (\ref{detail}) holds similar as in the usual HMC.

\vskip 0.1 in

\noindent
\textbf{Case 3}: $p_n<0$ and $T\ge n-N$ or $p_n\ge 0$ and $T\geq N+1-n$, and the dimension change has been accepted. The two figures shown in Case 3 will be identical, so we only consider the left one so that $p_n<0$ and $T\ge n-N$ and $M=N-1$. 

For convenience, we define $T_1=n-N$ and $T_2=T-n+N$. Denote $(\tilde{\bfq}_1^{N-1}, \tilde{\bfp}_1^{N-1})
=F_{-T_2}(\tilde \bfq^{N-1},\tilde \bfp^{N-1})$
and $(\bfq_1^N, \bfp_1^N):=F_{T_1}(\bfq^N, \bfp^N)$, the free energy gap is given by
\[
    \delta \cH^-=\Delta \cH^- -\frac{1}{\beta}\log(\pi_{N-1,N}((\bfq_1^N)|_{d_{N-1}+1}^{d_{N}}, (\bfp)_1^N|_{d_{N-1}+1}^{d_{N}}; \tilde{\bfq}_1^{N-1}, \tilde{\bfp}_1^{N-1})).
\]
provided that 
\[
(\tilde{\bfq}_1^{N-1}, \tilde{\bfp}_1^{N-1})
=(\bfq_1^N, \bfp_1^N)|_1^{d_{N-1}},
\]
or equivalently
\[
(\tilde \bfq^{N-1},\tilde \bfp^{N-1})
=F_{T_2}((\bfq_1^N, \bfp_1^N)|_1^{d_{N-1}})
\]
For this case to happen, the condition $K(p_n)\geq\delta\cH^-$ must hold. 
Hence, we have
\begin{multline*}
    \Gamma_T(\bfq^N,\bfp^N,n,p_n\to \tilde{\bfq}^{N-1}, \tilde{\bfp}^{N-1},\tilde n,\tilde p_n)\\
    =\mathbf{1}_{\{K(p_n)\geq\delta\cH^-\}}\delta ((\tilde \bfq^{N-1},\tilde \bfp^{N-1})
-F_{T_2}((\bfq_1^N, \bfp_1^N)|_1^{d_{N-1}}))\delta(n-\tilde n+\Delta n)\delta(\tilde p_n-p_n-\delta\cH^-)
\end{multline*}

For the reverse dynamics, we similarly denote
$(\tilde{\bfq}_1^{N-1}, -\tilde{\bfp}_1^{N-1})
=F_{T_2}(\tilde{\bfq}^{N-1}, -\tilde{\bfp}^{N-1})$. 
By the reversibility of the Hamilton dynamics, 
$(\tilde{\bfq}_1^{N-1}, -\tilde{\bfp}_1^{N-1})$ is the same as above. Similarly, we denote
$(\bfq_1^N, -\bfp_1^N):=F_{-T_1}(\bfq^N, -\bfp^N)$. 
Then, one has
\[
    \delta\cH^+=\Delta\cH^+ + \frac{1}{\beta}\log(\pi_{N-1,N}(\bfq_{new},-\bfp_{new}; \tilde{\bfq}_1^{N-1}, -\tilde{\bfp}_1^{N-1})), 
\]
provided that we have
\[
(\bfq^N, -\bfp^N)= F_{T_1}((\tilde{\bfq}_1^{N-1}, \bfq_{new}), (-\tilde{\bfp}_1^{N-1}, -\bfp_{new})).
\]

Hence, one has
\begin{multline*}
    \Gamma_T(\tilde \bfq^{N-1},-\tilde \bfp^{N-1},\tilde n,-\tilde p_n\to \bfq^N,-\bfp^N,n,-p_n)=\delta(\tilde n-n-\Delta n)
  \int  \mathbf{1}_{\{K(\tilde{p}_n)\geq\delta \cH^+\}}\pi_{N-1,N}\\
  \delta ((\bfq^N,-\bfp^N)-F_{T_1}((\tilde{\bfq}_1^{N-1}, \bfq_{new}), (-\tilde{\bfp}_1^{N-1}, -\bfp_{new})))\delta(p_n-\tilde p_n-\delta\cH^+)d\bfp_{new} d\bfq_{new}
\end{multline*}

Since $F_{T_1}$ is volume-preserving, we have
\begin{multline*}
\delta ((\bfq^N,-\bfp^N)-F_{T_1}((\tilde{\bfq}_1^{N-1}, \bfq_{new}), (-\tilde{\bfp}_1^{N-1}, -\bfp_{new})))\\
=\delta (F_{-T_1}(\bfq^N,-\bfp^N)-((\tilde{\bfq}_1^{N-1}, \bfq_{new}), (-\tilde{\bfp}_1^{N-1}, -\bfp_{new}))).
\end{multline*}
By the reversibility, the above is then reduced to
\begin{multline*}
\delta (F_{T_1}(\bfq^N,\bfp^N)-((\tilde{\bfq}_1^{N-1}, \bfq_{new}), (\tilde{\bfp}_1^{N-1}, \bfp_{new})))
=\\
\delta(F_{T_1}(\bfq^N,\bfp^N)|_1^{d_{N-1}}-(\tilde{\bfq}_1^{N-1}, \tilde{\bfp}_1^{N-1}))
\otimes \delta(F_{T_1}(\bfq^N,\bfp^N)|_{d_{N-1}+1}^{d_N}-(\bfq_{new}, \bfp_{new})).
\end{multline*}
Here, the two $\delta$'s on the right-hand side indicate the Dirac delta measures in $\R^{d_{N-1}}$ and $\R^{d_n-d_{N-1}}$ respectively. Hence, by evaluating the integral of $(\bfq_{new}, \bfp_{new})$, one finds that
\begin{multline*}
    \Gamma_T(\tilde \bfq^{N-1},-\tilde \bfp^{N-1},\tilde n,-\tilde p_n\to \bfq^N,-\bfp^N,n,-p_n)=\\
    \delta(\tilde n-n-\Delta n)
    \mathbf{1}_{\{K(\tilde{p}_n)>\delta \cH^+\}} \pi_{N-1, N}\delta(F_{T_1}(\bfq^N,\bfp^N)|_1^{d_{N-1}}-(\tilde{\bfq}_1^{N-1}, \tilde{\bfp}_1^{N-1}))\delta(p_n-\tilde p_n-\delta\cH^+)
\end{multline*}
where the $(\bfq_{new}, \bfp_{new})$ variables in the expressions of $\delta\cH^+$ and $\pi_{N-1,N}$ are replaced by $F_{T_1}(\bfq^N,\bfp^N)|_{d_{N-1}+1}^{d_N}$. 
 We also have that
\begin{equation*}
    \bar\pi(\bfq^{N},\bfp^{N},n,p_n)=\bar\pi(\tilde \bfq^{N-1},\tilde \bfp^{N-1},\tilde n,\tilde p_n)\pi_{N-1,N}=\bar\pi(\tilde \bfq^{N-1},-\tilde \bfp^{N-1},\tilde n,-\tilde p_n)\pi_{N-1,N}.
\end{equation*}
Noting also that $\delta\cH^+=-\delta\cH^-$ and the two indicator functions are also the same (either $1$ or $0$ and is determined by the starting and ending states), the detailed balance condition (\ref{detail}) holds.
\end{document}